\documentclass[1 [leqno,11pt]{amsart}
\usepackage{amssymb, amsmath,latexsym,amsfonts,amsbsy, amsthm,mathtools,graphicx,CJKutf8,CJKnumb,CJKulem,color,xcolor,mathrsfs,dsfont}
\usepackage{float}
\usepackage{hyperref}
\usepackage{braket}
\numberwithin{equation}{section}

\allowdisplaybreaks
\let\al=\alpha

\let\f=\frac

\let\pa=\partial

\def\R{\mathbb R}
\def\Z{\mathbb Z}
\def\T{\mathbb T}

\def\no{\noindent}

\def\oo{\infty}

\def\wkt{w_k^\mathcal{T}}
\def\wkit{w_{k,i}^\mathcal{T}}
\def\wkht{w_{k,h}^\mathcal{T}}
\def\wktt{w_k^{\widetilde{\mathcal{T}}}}
\def\wkitt{w_{k,i}^{\widetilde{\mathcal{T}}}}
\def\wkhtt{w_{k,h}^{\widetilde{\mathcal{T}}}}
\def\wkc{\hat{w}_k^\mathcal{C}}
\def\wkic{\hat{w}_{k,i}^\mathcal{C}}
\def\wkhc{\hat{w}_{k,h}^\mathcal{C}}

\def\oo{\infty}

\def\Ltr{{L^2_r}}
\def\Lor{{L^\oo_r}}
\def\Lpr{{L^p_r}}

\def\LtT{{L^2_T}}
\def\LoT{{L^\oo_T}}

\def\LpT{{L^p_T}}
\def\LtO{{L^2(\Omega)}}
\def\LpO{{L^p(\Omega)}}

\newcommand{\andf}{\quad\hbox{and}\quad}
\newcommand{\with}{\quad\hbox{with}\quad}

\newcommand{\beq}{\begin{equation}}
\newcommand{\eeq}{\end{equation}}
\newcommand{\ben}{\begin{eqnarray}}
\newcommand{\een}{\end{eqnarray}}
\newcommand{\beno}{\begin{eqnarray*}}
\newcommand{\eeno}{\end{eqnarray*}}

\newtheorem{theorem}{Theorem}[section]

\newtheorem{lemma}[theorem]{Lemma}
\newtheorem{proposition}[theorem]{Proposition}

\begin{document}
%\begin{CJK*}{UTF8}{gkai}

\title[Linear enhanced dissipation for the TC flow in the exterior region]
{Linear enhanced dissipation for the 2D Taylor-Couette flow in the exterior region: A supplementary example for Gearhart-Pr\"uss type lemma}

\author[T. Li]{Te Li}%
\address[T. Li]
 {Academy of Mathematics $\&$ Systems Science, The Chinese Academy of
Sciences, Beijing 100190, CHINA.}
\email{teli@amss.ac.cn}

\author[P. Zhang]{Ping Zhang}%
\address[P. Zhang]
 {Academy of Mathematics $\&$ Systems Science
and  Hua Loo-Keng Key Laboratory of Mathematics, The Chinese Academy of
Sciences, Beijing 100190, CHINA, and School of Mathematical Sciences, University of Chinese Academy of Sciences, Beijing 100049, CHINA. }
\email{zp@amss.ac.cn}

\author[Y. Zhang]{Yibin Zhang}%
\address[Y. Zhang]
 {Academy of Mathematics $\&$ Systems Science, The Chinese Academy of
Sciences, Beijing 100190, CHINA.  } \email{zhangyibin22@mails.ucas.ac.cn}

\date{\today}

\begin{abstract}From the perspective of asymptotic stability at high Reynolds numbers, Taylor-Couette flow, as a typical rotating shear flow, exhibits rich decay behaviors. Previously, for the extensively studied Couette flow or the Taylor-Couette flow in bounded annular domains, methods based on resolvent estimates could derive exponential decay asymptotic  for the solutions of the linearized system. However, unlike the Couette flow or the Taylor-Couette flow in bounded annular domains, the Taylor-Couette flow in exterior regions exhibits degeneration of derivatives of any order at infinity. In this paper, we present in Theorem \ref{main them-1} that the linearized system of the 2D Taylor-Couette flow in the exterior region exhibits space-time coupled  polynomial decay asymptotics. We also prove that the solution to this system, when it contains inhomogeneous terms, cannot be expected to exhibit space-time coupled exponential decay, as detailed in Theorem \ref{main them-2}. The result of Theorem \ref{main them-2} indicates that, even if we can obtain sharp resolvent estimates in different weighted spaces, the Gearhart-Pr\"uss type lemma no longer applies. This suggests that resolvent estimates may not be very effective for handling degenerate shear flows. Furthermore, Theorem \ref{main them-2} also implies that, for the transition threshold problem of the 2D Taylor-Couette flow in exterior regions, we can at most expect the solution to exhibit long-time behavior with space-time coupled polynomial decay.
Finally, we present a generalization of Theorem \ref{main them-2}, as detailed in Theorem \ref{main thm 3}.

\end{abstract}

\maketitle
\tableofcontents

\section{Introduction}
We consider 2D incompressible Navier-Stokes equations in the exterior region:
\begin{align}
\label{full nonlinear equation}\left\{
\begin{aligned}
&\partial_tv-\nu\Delta v+v\cdot \nabla v+\nabla p=0, \quad  (t,x) \in \R_+ \times \Omega_e ,\\
& \text{div } v=0,\\
&v(t,x)|_{t=0}=v_0(x),
\end{aligned}
\right.
\end{align}
where $\Omega_e=\bigl\{x=(r\cos \theta,r\sin\theta)\in \R^2:  |x|\geq 1\}= \{(r,\theta):  r \in [1,\oo), \theta\in\T\bigr\}$. The vector field $v(t,x)$ denotes the fluid velocity, $p(t,x)$ represents the scalar pressure function and the viscous coefficient $\nu>0$ is a constant. We denote the vorticity of $v$ by $\omega:=\partial_1v_2-\partial_2 v_1$, and focus on the stability of the 2D Taylor-Couette flow, in short TC flow,
\begin{align}\label{2D TC}
v^{*}(r,\theta)=(Ar+\f{B}{r})\left(
  \begin{array}{cc}
   -\sin\theta \\
   \cos\theta
    \end{array}
   \right),\quad  \omega^*(r,\theta) =2A,
\end{align}
 which is a steady solution to \eqref{full nonlinear equation}. Here $A,B$ are both constants.

The Taylor-Couette flow is a classical example of a simple system yet with complex and rich stability properties, as well as a prototype of anisotropic and inhomogeneous turbulence, as can be referred in such review books as \cite{AH,CI}. On more general grounds, the Taylor-Couette device is also an excellent prototype to study transport properties of most astrophysical or geophysical rotating shear flows \cite{DDDLRZ}: depending on the rotation speed of each cylinder, one can obtain various flow regimes with increasing or decreasing angular velocity and/or angular momentum. Regarding the issue of asymptotic stability at high Reynolds numbers \cite{BST}, Taylor-Couette flow, as a typical rotating shear flow, also provides rich decay characteristics. Therefore, investigating the stability of Taylor-Couette flow is a matter of great interest in both mathematics and physics. Earlier, the well-posedness and stability of the two-dimensional stationary Navier-Stokes equations in exterior regions \cite{BM,GHM} were widely studied and discussed, along with the recent research on the overall well-posedness and stability of small initial data \cite{M1,M2} in exterior regions.

Previously, the asymptotic stability problem of the 2D bounded region TC flow is considered in \cite{AHL-2}. This paper explores the enhanced dissipation phenomenon of the TC flow in the exterior region. Due to the degeneracy of all derivatives of the TC flow at infinity, this leads to stability phenomena that differ from those in the bounded region \cite{AHL-2}.

\subsection{Linearization near TC flow}\label{Taylor-Couette flow}
We first write \eqref{full nonlinear equation} in the vorticity formulation:
\begin{align}
\label{full nonlinear equation ns vor}
&\partial_t\omega-\nu\Delta\omega+v\cdot \nabla \omega=0.
\end{align}
By setting $w := \omega-\omega^*$, $u := v-v^{*}$, the vorticity perturbation equation near the 2D TC flow $(v^{*},\omega^*)$ is given by
\begin{align}\label{pertubation of the Taylor-Couette flow}
\left\{
\begin{aligned}
&\partial_tw-\nu\Delta w+\Bigl(A+\f{B}{r^2}\Bigr)\partial_{\theta}w+\f{1}{r}(\partial_r\varphi\partial_{\theta}w-\partial_{\theta}\varphi\partial_rw)=0,\\
&\Delta \varphi=w,  \qquad
(t,r,\theta)\in\mathbb{R}_{+}\times[1,\oo) \times\mathbb{T},
\end{aligned}
\right.
\end{align}
where $\Delta=\partial_r^2+r^{-1}\partial_r+r^{-2}\partial_{\theta}^2$ and $\varphi$ is the stream function. We implement \eqref{pertubation of the Taylor-Couette flow} with Navier-slip boundary conditions:
\begin{equation}\label{boundary conditions}
       w|_{\pa\Omega_e} =\varphi|_{\pa\Omega_e}=0.
\end{equation}

By expanding Fourier series of $w$ and $\varphi$ in $\theta$ variable:  $w=\sum\limits_{k\in\mathbb{Z}}e^{ik\theta}\hat{w}_k(t,r), \varphi=\sum\limits_{k\in\mathbb{Z}}e^{ik\theta}\hat{\varphi}_k(t,r),
$
and substituting them into  \eqref{pertubation of the Taylor-Couette flow}, and then comparing the Fourier coefficients, we find
\begin{align}\label{S1eq1}
\left\{
\begin{aligned}
&\partial_t\hat{w}_k-\nu\Bigl(\partial_{r}^2+\f{1}{r}\partial_{r}-\f{k^2}{r^2}\Bigr)\hat{w}_k+\Bigl(A+\f{B}{r^2}\Bigr)ik\hat{w}_k+ \hat{f}_k=0,\\
&\Bigl(\partial_r^2+\f{1}{r}\partial_r-\f{k^2}{r^2}\Bigr)\hat{\varphi}_k=\hat{w}_k,
\quad \quad
(t,r)\in\mathbb{R}_{+}\times[1,\oo),
\end{aligned}
\right.
\end{align}
where $\hat{f}_k=ir^{-1}\sum_{\ell\in\mathbb{Z}}\bigl((k-\ell)\partial_r\hat{\varphi}_\ell\hat{w}_{k-\ell}
-\ell\hat{\varphi}_{\ell}\partial_{r}\hat{w}_{k-\ell}\bigr)$.

To get rid of the terms in \eqref{S1eq1} which involve $\frac{\pa_r}{r}$ and $A$, we introduce
\begin{align*}
    \wkt:=r^{1/2}e^{ikAt}\hat{w}_k, \quad \varphi_k^{\mathcal{T}}:=r^{1/2} e^{ikAt}\hat{\varphi}_k,
    \quad f_k:=-r^{1/2}e^{ikAt}\hat{f}_k,\quad \Delta_{k,r}:=\partial_r^2-(k^2-1/4)/r^2.
\end{align*}
Then it follows from \eqref{S1eq1} that
\begin{subequations}\label{1.6}\begin{gather}
\label{scaling nonlinear}\partial_t \wkt + (- \nu\Delta_{k,r}+ikB /r^2)\wkt=f_k,\quad(t,r)\in\mathbb{R}_{+}\times[1,\oo),\\
\Delta_{k,r} \varphi_k^{\mathcal{T}} = \wkt, \\
\wkt|_{r=1} =\varphi_k^{\mathcal{T}}|_{r=1}=0.
\end{gather}\end{subequations}

We denote the $k$-mode linearized operator corresponding to the TC flow as
\begin{align}
\label{the linearized operator for tc} \mathcal{T}_k:= - \nu\Delta_{k,r}+ikB /r^2,
\end{align}
which acts on $L^2_r(1,\oo)$ with domain $D:=H^2\bigl((1,\oo),dr\bigr)\cap H^1_0\bigl((1,\oo),dr\bigr)$. For brevity, we may abbreviate $\mathcal{T}_k$ to $\mathcal{T}$, if there is no ambiguity.

Until now, we have transformed the problem into exploring the dynamic and long-time behavior of \eqref{1.6}. A standard decomposition yields two parts: the homogeneous linear equation and the inhomogeneous  equation with zero initial data, namely $\wkt=\wkht+\wkit$, where $\wkht$ and $\wkit$ satisfy the following two equations respectively:

 \begin{align}\label{homogenous eq tc}
        & \partial_t \wkht +\mathcal{T}\wkht = 0,\quad \wkht|_{t=0} =w_k(0), \quad \wkht|_{r=1,\infty}=0,\\
      \label{inhomogenous eq tc}  &
         \partial_t \wkit+\mathcal{T}\wkit = f_k, \quad \wkit|_{t=0}=0,\quad \wkit|_{r=1,\infty}=0.
\end{align}

\subsection{Notations and Main theorems}
We first list some commonly used notations. By $\Omega$, we mean an open domain in Euclidean space, which may differs from paragraph to  paragraph and depends on the context before and after. For $1\leq p\leq \oo $, we denote
\begin{equation*}
    \Lpr:= L^p\bigl((1,\oo), dr\bigr), \quad \LpO:= L^p(\Omega, dx), \quad
    \LpT := L^p\bigl((0,T), dt\bigr).
\end{equation*}
We define
\begin{align*}
\langle f,g\rangle_\Ltr=\int_1^\oo f(r)\overline{g}(r)dr,
\end{align*}
and $\braket{\cdot, \cdot}_\LtO$ is similarly defined.
 By $w'$, we mean $\pa_r w$. We denote $g\lesssim f$ if there exists a constant $C>0$ so that $g\leq C f$, moreover we denote $g \ll f$ if this $C$ is universal and small enough, and we use $g\lesssim_\al f$ to emphasize dependence of $C$ on $\al$. We always denote
\begin{align*}
    & \Delta_{k,r}:=\partial_r^2-(k^2-1/4)/r^2,\quad \mathcal{T}:= - \nu\Delta_{k,r}+ikB /r^2, \\
    &\qquad \qquad \quad    D= H^2\bigl((1,\oo),dr\bigr)\cap H^1_0\bigl((1,\oo),dr\bigr) ,\\
    &
\Phi(\mathcal{T})=\kappa_k= \nu^\f13 |kB|^\f23, \quad \mu_k=\max\{\nu k^2,\kappa_k \} ,\quad \Phi(\mathcal{C})=\nu^\f13 |k|^\f23.
\end{align*}

Below we present our main theorems. In Section \ref{motivations}, we shall explain the motivations and their connections to  Gearhart-Prüss type lemma and enhanced dissipation of transport-diffusion equations.
\begin{theorem} \label{main them-1}
   Let $k\in \mathbb{Z}\backslash\{0\}$,  $q\in \mathbb{N}$, $\frac{\nu}{|kB|}\ll (1+q)^{-3}$, $T>0$ and let $\wkht$ be the solution to  \eqref{homogenous eq tc} with $w_k(0)\in \Ltr$. There exists constant $C=C(q)>0$ independent of $\nu,k,B,T$ and $w_k(0)$, so that
\begin{equation*}
   %\label{pointwise enhanced dissipation tc-main them-1}  
   \bigl\|(1+\Phi(\mathcal{T}) t/r^2)^{q} \wkht\bigr\|_{\LoT(\Ltr)} + \Phi(\mathcal{T})^{1/2} \bigl\| (1+\Phi(\mathcal{T}) t/r^2)^{q} r^{-1} \wkht\bigr\|_{\LtT(\Ltr)} \leq C \|w_k(0)\|_{\Ltr}.
\end{equation*}
\end{theorem}

\begin{theorem} \label{main them-2}
For any positive continuous functions \( a_1(r) \), \( a_2(r) \) and strictly decreasing function \( \phi(r) \), there exists a sequence of functions,
 \(\{(w_n, f_n)\}\), that solves \eqref{inhomogenous eq tc} as below
\begin{align*}
  \partial_t w_n +\mathcal{T}w_n=f_n, \quad w_n|_{t=0}=0, \quad  w_n|_{r=1,\infty}=0,
\end{align*}
and satisfies for any $n \in \mathbb{N}$
\begin{align}
\label{finite energy of w_n, f_n}\|e^{t \phi(r)} a_1(r)  w_n\|_{L^2((0,\oo); \Ltr)} + \|e^{t\phi(r)} a_2(r) f_n\|_{L^2((0,\oo); \Ltr)} < +\oo.
\end{align}
However, there holds
\begin{equation}\label{false main them-2}
    \limsup_{n \rightarrow +\oo} \frac{\|e^{t\phi(r)} a_1(r)  w_n\|_{L^2((0,\oo); \Ltr)}}{ \|e^{t\phi(r)} a_2(r) f_n\|_{L^2((0,\oo); \Ltr)}} = +\oo.
\end{equation}
\end{theorem}

The proof of Theorem \ref{main them-1} highly relies on the transport term $\frac{ik B}{r^2}$ in $\mathcal{T}$, while the proof of Theorem \ref{main them-2} doesn't. Indeed, the proof of Theorem \ref{main them-2} only relies on the forward translation invariance in time of the equation. Hence along the same lines and under the same assumptions, same results in Theorem \ref{main them-2}  also hold for {\bf Examples E1} and {\bf E2} in Section \ref{motivations}. This motivates us to derive a generalization of Theorem  \ref{main them-2}.

We consider an abstract evolution equation
\begin{equation}\label{abstracted evolution equation}
         \partial_tw+\mathcal{L}w = f,\quad w|_{t=0}=w(0,x), \quad (t,x) \in \R_+ \times \Omega,
\end{equation}
where $\Omega \subset \R^n$ and \( \mathcal{L} \) is a closed, densely defined linear operator on $L^2(\Omega)$ with domain $D(\mathcal{L})$, which generates a $\LtO$-strongly continuous semigroup $\{e^{t\mathcal{L}}\}_{t\geq0}$.

\begin{theorem}\label{main thm 3}
    We assume that \eqref{abstracted evolution equation} has fundamental solution $p_\mathcal{L}(t,x,y) \in \mathcal{C}^\oo(\R_+\times \Omega\times \Omega)$, in the sense that, for any $w(0,x)\in \mathcal{C}^\oo_c(\Omega)$ and $f(t,x)\in \mathcal{C}_c^\oo(\R_+\times \Omega)$, the mild solution of \eqref{abstracted evolution equation} can be  represented as
\begin{equation}\label{duhamel formula}
    w(t,x) = \int_\Omega p_\mathcal{L}(t,x,y) w(0,y) dy + \int_0^t \int_\Omega p_\mathcal{L}(t-s,x,y) f(s,y)dy ds.
\end{equation}
We assume in addition that the support of $ p_\mathcal{L}$ is $\R_+ \times \Omega\times \Omega$.
Then for any positive continuous functions \( a_1(x) \), \( a_2(x) \) and non-constant continuous function \( \phi(x) \), there exists a sequence of functions, \(\{(w_n, f_n)\}\), that solves \eqref{duhamel formula} with $w_n(0,x)=0$,
and satisfies for any $n \in \mathbb{N}$,
\begin{align}\label{1.15}
f_n(t,x)\in \mathcal{C}_c^\oo(\R_+\times \Omega),\quad \|e^{t\phi(x)} a_2(x) f_n\|_{L^2((0,\oo)\times \Omega)} < +\oo.
\end{align}
However, there holds
\begin{equation}\label{false main thm 3}
    \limsup_{n \rightarrow +\oo} \frac{\|e^{t\phi(x)} a_1(x)  w_n\|_{L^2((0,\oo)\times \Omega)}}{ \|e^{t\phi(x)} a_2(x) f_n\|_{L^2((0,\oo)\times \Omega)}} = +\oo.
\end{equation}
\end{theorem}

\subsection{The motivations of this work}\label{motivations}
\subsubsection{Enhanced dissipation of transport-diffusion equations}
One motivation of this work stems from the study of transport-diffusion equation
\begin{equation}\label{1.16}
    \pa_t w - \nu \Delta w +V\cdot
    \nabla w=0.
\end{equation}
Here  the constant $\nu\ll 1$ designates the viscous coefficient.
If we restrict
$$V(x,y)=\begin{pmatrix} s(y)\\0 \end{pmatrix} \quad \text{for some} \ s(y) \  \text{smooth enough},$$
and $(x,y)\in \T\times \Omega$, where $\Omega=\T$ or $[0,1]$.   Then \eqref{1.16} reads
\begin{align}
\label{transport-diffusion equation}\left\{
\begin{aligned}
&\partial_t w -\nu (\pa_x^2+\pa_y^2) w + s(y) \pa_x w=0, \quad  (t,x,y) \in \R_+ \times \T \times \Omega ,\\
&w(t)|_{t=0}=w(0).
\end{aligned}
\right.
\end{align}
 When $\Omega=[0,1]$, we implement \eqref{transport-diffusion equation} with Dirichlet or Neumann Boundary conditions. We write $w(t,x,y)=\sum_{k\in\Z} \hat{w}_k(t,y) e^{ikx}$ and then obtain the $k$-mode equation
\begin{align}
\label{k-mode transport-diffusion equation}\left\{
\begin{aligned}
&\partial_t \hat{w}_k -\nu (\pa_y^2-k^2) \hat{w}_k + iks(y)\hat{w}_k= 0, \quad  (t,y) \in \R_+  \times \Omega ,\\
&\hat{w}_k(t)|_{t=0}=\hat{w}_{k}(0).
\end{aligned}
\right.
\end{align}

When $s(y)=0$, \eqref{transport-diffusion equation} and \eqref{k-mode transport-diffusion equation} are actually heat equations. Hence we can only expect $\hat{w}_k(t)$ decays after $t\geq \nu^{-1}$. However if $s(y)$ has a finite number of critical points, of maximal order $n_0$ (i.e., for any $y\in\Omega$, there exists $n\leq n_0$ so that $s^{(n+1)}(y)\neq 0$), then for $k\neq 0$, $\hat{w}_k$ will decay when $t\ll \nu^{-1}$. This phenomenon is called ``enhanced dissipation" and could date back to Hörmander's hypoellipticity theory. We recall the following theorem for readers' convenience:
\begin{theorem}[\cite{BZ17,ABN22}]\label{thm1.4}
    There exist constants $C,c,\nu_0>0$ depending only on $s(y)$, so that for any $\nu\leq \nu_0$ and $t\geq0$, there holds
\begin{equation}
    \|\hat{w}_k(t)\|_{L^2(\Omega)} \leq Ce^{-c\nu^{\frac{n_0+1}{n_0+3}}t} \|\hat{w}_k(0)\|_{L^2(\Omega)}.
\end{equation}
\end{theorem}

Similarly, if we restrict
$$V(r,\theta)=rs(r) \mathbf{e}_\theta, \quad (r,\theta)\in \Omega \times  \T \subset [1,\oo)\times \T,$$
and write $w(t,r,\theta)=\sum_{k\in\Z}  \hat{w}_k(t,r) e^{ik\theta}=\sum_{k\in \Z} r^{\f12} w_k(t,r) e^{ik\theta}$, we find
\begin{align}
\label{k-mode transport-diffusion equation,rotational type}\left\{
\begin{aligned}
&\partial_t w_k -\nu \Bigl(\pa_r^2-\frac{k^2-\f14}{r^2}\Bigr) w_k + iks(r)w_k= 0, \quad  (t,r) \in \R_+  \times \Omega ,\\
&w_k(t)|_{t=0}=w_{k}(0), \quad w_k|_{\pa\Omega}=0.
\end{aligned}
\right.
\end{align}
Then the same result as Theorem \ref{thm1.4} holds with $n_0$ similarly defined, see for example \cite{CDL24,FMN23}.

\no {\bf Example E1.} Couette flow on bounded domain, see \cite{CLWZ-2D-C}.

Let $\Omega=[0,1], s(y)=y$ and $\hat{w}_k|_{y=0,1}=0$. It's easy to check that $n_0=0$ and thus we have
\begin{equation*}
     \|\hat{w}_k(t)\|_{L^2(\Omega)} \leq Ce^{-c\nu^{\f13}|k|^{\f23}t} \|\hat{w}_k(0)\|_{L^2(\Omega)}.
\end{equation*}

\no {\bf Example E2.} TC flow on bounded annular domains, see \cite{AHL-2}.

 Let $\Omega=[1,R], s(r)=\frac{B}{r^2}$ and $w_k|_{r=1,R}=0$. It's easy to check that $n_0=0$ and thus we have
\begin{equation*}
     \|w_k(t)\|_{L^2(\Omega)} \leq C e^{-c\nu^{\f13}|kB|^{\f23}t/R^2} \|w_k(0)\|_{L^2(\Omega)}.
\end{equation*}

We shall go deeper into these examples in Subsection \ref{1.4}, but now let's investigate \eqref{homogenous eq tc} by virtue of Theorem \ref{thm1.4}. It's easy to see that $\Omega=[1,\oo)$ and $s(r)=\frac{B}{r^2}$ is infinitely degenerate (i.e., $n_0=\oo$). Thus, it seems impossible to obtain
\begin{equation*}\label{pointwise enhanced dissipation}
    \|\wkht(t)\|_{L^2_r} \lesssim e^{-c\nu^{\varsigma}t} \|\wkht(0)\|_{L^2_r}, \quad \text{for some  } \ 0<\varsigma < 1.
\end{equation*}

It's natural to ask: if we replace $\Ltr$ by two different functional  spaces $Y,Z$, could we recover
\begin{equation}\label{1.21}
    \|\wkht(t)\|_{Y} \lesssim \varrho(\nu^\varsigma t) \|\wkht(0)\|_{Z},
\end{equation}
for some  \(0<\varsigma < 1\), and   strictly decreasing function \(\varrho(t)\) that approaches $0$ as $t$  tends to $\infty$?

The answer is positive! A simple corollary of Theorem \ref{main them-1} yields that, for any $t\geq 0$,
\begin{equation*}
    \bigl\|\bigl(1+\frac{\nu^\f13|kB|^\f23 t}{r^2}\bigr) \wkht\bigr\|_{\Ltr} \lesssim \|\wkht(0)\|_\Ltr,
\end{equation*}
which leads to
\begin{equation*}
    \|\wkht\|_{L^2((1,\oo), r^{-4}dr)}  \lesssim \frac{1}{\bigl(1+ \nu^\f13|kB|^\f23 t\bigr)} \|\wkht(0)\|_\Ltr.
\end{equation*}

This is a simple example of \eqref{1.21} with $\varsigma=\f13$, $Y={L^2\bigl((1,\oo), r^{-4}dr\bigr)}$ and $Z=\Ltr$, which indicates that for \eqref{homogenous eq tc}, we can recover an enhanced dissipation scale $\mathcal{O}(\nu^{-\f13}|kB|^{-\f23})$ at the expense of losing spatial weights!

\subsubsection{Gearhart-Prüss type lemma and its applications.}\label{1.4}
We now turn to \eqref{abstracted evolution equation}. A typical decomposition gives $w=w_{h}+w_i$, where
\begin{subequations}
\begin{gather}
   \label{1.19a} \pa_t w_h + \mathcal{L}w_h = 0, \quad w_h(0)=w(0) ,\\
   \label{1.19b}  \pa_t w_i + \mathcal{L}w_i = f, \quad w_i(0)=0 .
\end{gather}
\end{subequations}
As in \cite{Pa}, we call an operator $\mathcal{L}$ in a Hilbert space $H$ to be accretive if
\begin{align}
\label{accretive}\Re\langle \mathcal{L}w, w\rangle_H \geq0,\quad \textrm{for any}\ w\in D(\mathcal{L}).
\end{align}
The operator $\mathcal{L}$ is said to be m-accretive if, in addition, any  $\lambda<0$ belongs to the resolvent set of $\mathcal{L}$ (see \cite{Kato} for more details). The pseudospectral bound of $\mathcal{L}$ is defined as
\begin{align}\label{pseudospectral bound}
\Psi(\mathcal{L}):=\inf\bigl\{\|(\mathcal{L}-i\lambda)w\|_H :w\in D(\mathcal{L}),\ \lambda\in\mathbb{R},\ \|w\|_H=1\bigr\}.
\end{align}

To obtain decaying semigroup bounds from pseudospectral bounds, let us recall the following Gearhart-Prüss type lemma established by Wei in \cite{W21}, (Helffer and Sjöstrand also provide related results in \cite{HS}).

\begin{lemma}[\cite{W21}] \label{GP lemma}Let $\mathcal{L}$ be a m-accretive operator in a Hilbert space $H$ and let $w_h,w_i$ be the solution of \eqref{1.19a}, \eqref{1.19b} respectively. Then there hold
\begin{subequations}
    \begin{gather}
    \|w_h(t)\|_H = \|e^{-t\mathcal{L}} w(0)\|_H\leq e^{-t\Psi(\mathcal{L})+\f{\pi}{2}} \|w(0)\|_H, \quad \textrm{for any} \ t\geq0,\label{1.21 a} \\
     \Psi(\mathcal{L})\|w_i\|_{\LtT(H)} \leq \|f\|_{\LtT(H)} , \quad \textrm{for any} \ T > 0 \label{1.21 b}.
    \end{gather}
\end{subequations}
\end{lemma}

Below let's go back to {\bf Examples E1, E2}, and review their relationships with Lemma \ref{GP lemma}.

\no{\bf Example E1}. Couette flow.

We first review the relevant results regarding 2D Couette flow in $\mathbb{T}\times\mathbb{R}$ and $\mathbb{T}\times[0,1]$. Let us  recall  the linearized equation near the 2D Couette flow \cite{BGM-bams,BH,BMV,BWV,CLWZ-2D-C} based on the Navier-Stokes equations as follows:
\begin{align*}
\left\{
\begin{aligned}
 &\partial_tw^{\mathcal{C}}- \nu(\partial_x^2+\partial_y^2)w^{\mathcal{C}}+y\partial_xw^{\mathcal{C}} = f, \quad (t,x,y)\in \R_+ \times\mathbb{T} \times\Omega,\\
 & w^\mathcal{C}|_{t=0} =w(0),\quad w^\mathcal{C}|_{\pa \Omega}=0,
\end{aligned}
\right.
\end{align*}
where $\Omega=\mathbb{R}$ or $[0,1]$. By setting $\hat{w}^\mathcal{C}=\sum_{k\in \Z}\wkc e^{ikx}$ and $f=\sum_{k\in \Z}\hat{f}_k e^{ikx}$, we obtain the following $k$-mode linearized equation:
\begin{align*}
\left\{
\begin{aligned}
 & \partial_t\wkc- \nu(\partial_y^2-k^2)\wkc+iky\wkc = \hat{f}_k, \quad (t,y)\in \R_+ \times \Omega,\\
 & \wkc|_{t=0} =\hat{w}_k(0),\quad \wkc|_{\pa \Omega}=0.
\end{aligned}
\right.
\end{align*}
We denote
\begin{equation}\begin{split}
\label{the linearized operator for c} &\qquad \quad  \mathcal{C}:= - \nu(\partial_y^2-k^2)+iky,\quad H:=\LtO,\\
& D(\mathcal{C}):= \bigl\{w\in \LtO: w\in H^2(\Omega)\cap H^1_0(\Omega),\  yw \in \LtO \bigr\},
\end{split}\end{equation}
and decompose \(\wkc=\wkhc + \wkic \) as in \eqref{1.19a} and \eqref{1.19b}. It's easy to check that $\mathcal{C}$ is m-accretive in $H$.

Below we recall the resolvent estimate corresponding to \(\mathcal{C}\):
\begin{lemma}[\cite{CLWZ-2D-C}] \label{resolvent estimate-c}
For any $k\in \Z\backslash\{0\}$, $|k|\gg \nu>0$ and $\lambda\in\mathbb{R}$, there exists constant $C>0$  independent of $\nu,k,\lambda$, so that
\begin{align*}
   (\nu k^2)^{1/3} \|w\|_{\LtO} \leq C \|(\mathcal{C}-i\lambda)w\|_{\LtO}.
\end{align*}
\end{lemma}

\begin{proof}
    The  details can be found directly in Lemma \ref{eq:R-w-navier} in the appendix.
\end{proof}

Actually by combining Lemma \ref{resolvent estimate-c} with Lemma \ref{eq:R-w-navier-2}, it could be shown that
$ (\nu k^2)^{1/3}$ is  optimal for low frequency $k$, i.e.,
\begin{align}\label{LC}
\Psi(\mathcal{C})\sim  (\nu k^2)^{1/3}=:\Phi(\mathcal{C}).
\end{align}
With Lemmas \ref{GP lemma} and \ref{resolvent estimate-c} at hand, we then obtain:
\begin{lemma}[\cite{CLWZ-2D-C}]
    For any $k\in \Z\backslash\{0\}$, $|k|\gg \nu>0$, $T>0$, there exist constants $C,c>0$  independent of $\nu,k,T$, so that
\begin{subequations}
    \begin{gather}
\label{1.29a}\bigl\|e^{c\Phi(\mathcal{C}) t} \wkhc\bigr\|_{\LoT(\LtO)} + \Phi^{1/2}(\mathcal{C}) \bigl\|e^{c\Phi(\mathcal{C}) t} \wkhc\bigr\|_{\LtT(\LtO)} \leq C \|\hat{w}_k(0)\|_{\LtO},\\
\Phi(\mathcal{C}) \bigl\|e^{c\Phi(\mathcal{C}) t} \wkic\bigr\|_{\LtT(\LtO)} \leq C \bigl\|e^{c\Phi(\mathcal{C}) t}\hat{f}_k\bigr\|_{\LtT(\LtO)}.
    \end{gather}
\end{subequations}
\end{lemma}
\begin{proof}
    It follows from Lemma \ref{GP lemma} and Lemma \ref{resolvent estimate-c}. One may see \cite{CLWZ-2D-C} for details.
\end{proof}

\no{\bf Example E2.} TC flow on bounded annular domains.

In this example, we review the relevant results regarding 2D TC flow on bounded annular domains $\{x\in \R^2: 1\leq |x| \leq R<\oo\}$, instead of exterior domain $\Omega_e$. We recall from \cite{AHL-2} the following $k$-mode linearized equation:
\begin{align*}
\left\{
\begin{aligned}
 & \partial_t\wktt+ \widetilde{\mathcal{T}}\wktt = f_k, \quad (t,r)\in \R_+ \times \Omega,\\
 & \wktt|_{t=0} =w_k(0),\quad \wktt|_{\pa \Omega}=0,
\end{aligned}
\right.
\end{align*}
where $\Omega= [1,R]$ and we denote
\begin{equation*}\begin{split}
& \widetilde{\mathcal{T}}= - \nu\bigl(\partial_r^2-(k^2-1/4)/r^2\bigr)+ikB /r^2, \quad H:=\LtO, \quad D(\widetilde{\mathcal{T}}):=  H^2(\Omega)\cap H^1_0(\Omega).
\end{split}\end{equation*}

\begin{lemma}[\cite{AHL-2}]\label{resolvent estimate-t-bounded domain}
For any $k\in \mathbb{Z}\backslash\{0\}$, $|kB|\gg \nu >0$, $1<R <\oo$ and $\lambda\in\mathbb{R}$, there exists constant $C>0$  independent of $\nu,k,B,\lambda,R$, so that
\begin{align}
 \label{resolvent estimate-tc-bounded domain} \nu^{1/3} |kB|^{2/3}\|w/r\|_{\LtO} \leq C \|r(\widetilde{\mathcal{T}}-i\lambda)w\|_{\LtO}.
\end{align}
\end{lemma}
\begin{proof}
Details can be found in Proposition 3.2 from \cite{AHL-2}.
\end{proof}
We decompose $\wktt=\wkhtt+\wkitt$ as in \eqref{1.19a}, \eqref{1.19b} and denote $\tilde{X}_1:=L^2\bigl((1,R),rdr\bigr)$, $\tilde{X}_2:=L^2\bigl((1,R),\frac{dr}{r}\bigr)$. In view of Lemmas \ref{resolvent estimate-t-bounded domain} and \ref{resolvent estimate-2},  the pseudospectral bound $\nu^{1/3} |kB|^{2/3}$ is sharp for low frequency $k$, i.e.,
\begin{align}\label{LT-bounded domain}
\Psi\bigl(\widetilde{\mathcal{T}}(\tilde{X}_1\to \tilde{X}_2)\bigr)\sim  \nu^{1/3} |kB|^{2/3}=:\Phi(\widetilde{\mathcal{T}}).
\end{align}
Since on $[1,R]$, $\tilde{X}_1=\tilde{X}_2=\LtO$,  we can still use $\Phi(\widetilde{\mathcal{T}})$  to characterize the decay of $\wkhtt$ and $\wkitt$. That is:
\begin{lemma}[\cite{AHL-2}]
  For any $k\in \mathbb{Z}\backslash\{0\}$, $|kB|\gg \nu >0$, $1<R <\oo$ and $\lambda\in\mathbb{R}$, there exist constants $C,c>0$  independent of $\nu,k,B,\lambda,R$, so that
\begin{subequations}
\begin{gather}
   \label{pointwise enhanced dissipation tc-bounded domain}
    \bigl\|e^{\frac{c}{R^2}\Phi(\widetilde{\mathcal{T}}) t} \wkhtt\bigr\|_{\LoT(\LtO)} + R^{-1}\Phi^{1/2} (\widetilde{\mathcal{T}}) \bigl\|e^{\frac{c}{R^2}\Phi(\widetilde{\mathcal{T}}) t} \wkhtt\bigr\|_{\LtT(\LtO)} \leq C \|w_k(0)\|_{\LtO},\\
    \label{integrated enhanced dissipation tc-bounded domain}   R^{-1}\Phi(\widetilde{\mathcal{T}})
    \bigl\|e^{\frac{c}{R^2}\Phi(\widetilde{\mathcal{T}}) t} \wkitt\bigr\|_{\LtT(\LtO)} \leq C R\|e^{\frac{c}{R^2}\Phi(\widetilde{\mathcal{T}}) t}f_k\|_{\LtT(\LtO)}.
\end{gather}
\end{subequations}
\end{lemma}
\begin{proof}
Details can be found in Propositions 4.3 and 4.5 from \cite{AHL-2}.
\end{proof}

Now for TC flow on exterior domain, $\Omega=(1,\oo)$, i.e., $R=\oo$ in {\bf Example E2}. In this case, we still get, by
using Lemma \ref{resolvent estimate-1}, that
\begin{align}\label{LT-extrior domain}
\Psi\bigl(\mathcal{T}(X_1\to X_2)\bigr)\sim  \nu^{1/3} |kB|^{2/3}=:\Phi(\mathcal{T}),
\end{align}
where $X_1 := L^2\big((1,\oo),rdr\big),X_2 := L^2\big((1,\oo),\frac{dr}{r}\big)$. However in this case $X_1\neq X_2$, thus Lemma \ref{GP lemma} is no longer applicable. But in view of \eqref{pointwise enhanced dissipation tc-bounded domain} and \eqref{integrated enhanced dissipation tc-bounded domain}, it's natural to ask the following two questions:
\begin{itemize}
     \item[{\bf Q1}:] Does there exist some \( c > 0 \) so that
\begin{align}
  \label{Q1-1}  \bigl\|e^{c\Phi(\mathcal{T}) t/r^2} \wkht\bigr\|_{\LoT(\Ltr)} + \Phi^{1/2}(\mathcal{T}) \bigl\|r^{-1}e^{c\Phi(\mathcal{T}) t/r^2} \wkht\bigr\|_{\LtT(\Ltr)} \leq C \|w_k(0)\|_{\Ltr}?
\end{align}

   \item[{\bf Q2}:] Does there exist some \( c > 0 \) so that
\begin{align}
   \label{Q2-1} \Phi(\mathcal{T}) \bigl\|r^{-1}e^{c\Phi(\mathcal{T}) t/r^2} \wkit\bigr\|_{\LtT(\Ltr)} \leq C \|e^{c\Phi(\mathcal{T}) t/r^2}rf_k\|_{\LtT(\Ltr)}?
\end{align}
 \end{itemize}

Another motivation of this paper is to answer {\bf Q1} and {\bf Q2}.  We don't have an exact answer to {\bf Q1}. However if we take $a_1(r)=\frac{1}{r}$, $a_2(r)=r$ and $\phi(r)=\frac{c \Phi(\mathcal{T})}{r^2}$ in Theorem \ref{main them-2}, we find that {\bf Q2} is false!
So up to this point, we can view Theorems \ref{main them-1} and \ref{main them-2} as a supplementary example for the Gearhart-Pr\"uss type lemma \ref{GP lemma}.

\subsubsection{One-dimensional Schrödinger operators $\mathcal{L}:=-\Delta+V$ on $L^2(\Omega)$ and  $e^{-t\mathcal{L}}$.}

 Nonnegative Schrödinger operators \( \mathcal{L} \) on \( L^2(\mathbb{R}^N) \), and their associated heat semigroups \( e^{-t\mathcal{L}} \) have been extensively studied by many mathematicians since Simon's pioneering work \cite{Simon}. See also the monographs of Davies \cite{Davies}, Grigor’yan \cite{Grigoryan}, and Ouhabaz \cite{Ouhabaz}, which provide further insights into this topic.

In this paper, the potential \( V \) in the one-dimensional Schrödinger operator $\mathcal{L}$ is purely imaginary, so the Schrödinger operator no longer possesses the nonnegative property. However, as a substitute for nonnegativity, we have already mentioned in the subsection \ref{1.4} that $\mathcal{L}$ needs to be an accretive operator of the form \eqref{accretive}. It is clear that when \( V \) is purely imaginary, this Schrödinger operator is accretive. 

The behavior of the semigroups \( e^{-t\mathcal{L}} \)-whether they exhibit exponential or polynomial decay-depends not only on whether $\Omega$ is compact, but also on the properties of the potential function \( V \) itself. Therefore, estimating \( e^{-t\mathcal{L}} \) is a challenging yet very interesting mathematical problem. Here are a few examples.

\no{\bf Example 1}. $\mathcal{L}=-\partial_y^2$ with $\Omega=\R$ or $\Omega=[1,\infty)$. \( e^{-t\mathcal{L}} \) does not exhibit exponential decay. This result is classical, for specifics, see Example 4.1 and Example 4.2 in Section \ref{section 4}.

\no{\bf Example 2}. $\mathcal{L}=-\partial_y^2$ with $\Omega$ is bounded. \( e^{-t\mathcal{L}} \) exhibits exponential decay. This can be derived from the Poincaré inequality.

\no{\bf Example 3}. Couette flow, $\mathcal{L}=-\partial_y^2+iy$ with $\Omega=\R$, $\Omega=[0,1]$ or $\Omega=[1,\infty)$. Regardless of whether \( \Omega \) is bounded or unbounded, \( e^{-t\mathcal{L}} \) exhibits exponential decay \eqref{1.29a}. Therefore, the purely imaginary potential alters the properties of the Laplace operator.

\no{\bf Example 4}. TC flow, $\mathcal{L}=-\partial_y^2+i\f{1}{y^2}$ with $\Omega=[1,R]$. \( e^{-t\mathcal{L}} \) exhibits exponential decay \eqref{pointwise enhanced dissipation tc-bounded domain}. However, for the case where \( \Omega=[1,\infty) \) is an exterior domain, we can currently only obtain polynomial decay estimates as stated in Theorem \ref{main them-1}, and we are unable to answer whether \( e^{-t\mathcal{L}} \) exhibits exponential decay.

\subsubsection{Transition threshold.}From the perspective of linear systems, this paper aims to address whether the Gearhart-Pr\"uss type lemma remains applicable to the 2D TC flow in the exterior region, specifically whether {\bf Q1} and {\bf Q2} still hold, which is approached from a functional analysis standpoint. In fact, our final goal is to investigate the nonlinear stability of the 2D TC flow in the exterior region as below
\begin{align*}
\partial_t \wkt + (- \nu\Delta_{k,r}+ikB /r^2)\wkt=f_k,\quad(t,r)\in\mathbb{R}_{+}\times[1,\oo),
\end{align*}
especially the transition threshold problem for this flow, which will be appeared in our subsequent paper \cite{LZZ-2D-TC}. Therefore, addressing the questions of {\bf Q1} and {\bf Q2} also addresses whether the pointwise estimates of the nonlinear system can still exhibit exponential decay. The negative answer of {\bf Q2} indicates that the pointwise estimates for the nonlinear system of the 2D TC flow in the exterior region do not possess exponential decay but at most  polynomial decay.

\section{Space-time estimates for $\wkt$}
\subsection{Resolvent estimates}
In order to establish the decay properties of the linearized equation \eqref{scaling nonlinear}, we begin with studying the fundamental resolvent equation
\begin{align}\label{vorticity eqn}
& (\mathcal{T}-ikB\lambda) w=F=F_1+\pa_r F_2 \quad \textrm{with} \ w\in D,
\end{align}
where $\lambda\in\mathbb{R}$ and $k\in \Z\backslash\{0\}$. By a density argument in $\Ltr$, we may assume in addition in this section that $w,F_1, F_2\in \mathcal{C}^\oo_c([1,\oo))$,

\begin{lemma}\label{trivial w' lemma}For any $w\in D$, there holds
  \begin{align}
\label{trivial w'}&\Re\langle F,w\rangle_\Ltr=\nu\|w'\|_{\Ltr}^2+\nu(k^2-1/4)\|w/r\|_{\Ltr}^2.
\end{align}
\end{lemma}
\begin{proof}It follows by integration by parts.
\end{proof}

To proceed, we fix some decreasing function $\rho\in \mathcal{C}^\oo(\R)$ satisfying
$$\rho(z)=1, \ \text{for} \ z\leq -1, \qquad  \rho(z)=-1, \ \text{for} \ z\geq 1.$$

\begin{lemma}\label{resolvent estimate-1}For any $k\in \mathbb{Z}\backslash\{0\}$, $|kB|\gg \nu >0$ and $\lambda\in\mathbb{R}$, there exists constant $C>0$  independent of $\nu,k,B,\lambda$, so that
\begin{align}\label{resolvent estimate-0-1}\nu\|w'\|_{\Ltr}^2+\mu_k\|w/r\|_{\Ltr}^2\leq C\bigl(\bigl|\langle F,w\rangle_\Ltr\bigr|+ \bigl|\langle F,\tilde{\rho} w\rangle_\Ltr\bigr|\mathds{1}_{(0,1)}(\lambda) \bigr),
\end{align}
where $\mu_k=\max\bigl\{\nu^{1/3} |kB|^{2/3},\nu k^2\bigr\}$, $\tilde{\rho} := \rho\Bigl(\frac{r-\lambda^{-\f12}}{\nu^\f13|kB|^{-\f13}\lambda
^{-\f12}}\Bigr)$ for $\lambda\in(0,1)$.

Moreover, there holds
\begin{equation}\label{2.4}
    \begin{split}
\nu^\f12 \|w'\|_\Ltr + \mu_k^\f12 \|w/r\|_\Ltr \lesssim \mu_k^{-\f12} \|rF_1\|_\Ltr +\nu^{-\f12}\|F_2\|_\Ltr.
    \end{split}
\end{equation}
\end{lemma}
\begin{proof}
The discussion is divided into the following three cases:
\begin{align*}
\lambda\in(-\infty,0],\quad \lambda\in[1,\infty),\quad \lambda\in(0,1).
\end{align*}

\no\textbf{(1). Case of $\lambda\in(-\infty,0]$.}

It is easy to observe that
\begin{align*}
\big|\Im\langle F,w\rangle_\Ltr\big|=|kB|\big|\langle(r^{-2}-\lambda)w,w\rangle_\Ltr\big|\geq |kB|\|w/r\|_{\Ltr}^2,
\end{align*}
which gives
\begin{align}
\label{lmabda trival-0}\bigl|\langle F,w\rangle_\Ltr\bigr|\geq|kB|\|w/r\|_{\Ltr}^2.
\end{align}

\no \textbf{(2). Case of $\lambda\in[1,\infty)$.}

Since $\lambda\in[1,\infty)$, there holds
\begin{align*}
\big|\Im\langle F,w\rangle_\Ltr\big|=|kB|\big|\langle(r^{-2}-\lambda)w,w\rangle_\Ltr\big|\geq |kB|\langle(1-r^{-2})w,w\rangle_\Ltr,
\end{align*}
which implies
\begin{align*}
&|kB|\int_1^{\infty}(1-r^{-2})|w|^2dr\leq\bigl|\langle F,w\rangle_\Ltr\bigr|.
\end{align*}
Let $0<\delta\ll1$ be a small constant to be determined later, we find
\begin{align*}
    &\int_1^{\infty}(1-r^{-2})|w|^2dr\geq \int_{1+\delta}^{\infty}(1-r^{-2})|w|^2dr\\
    &=\int_{1+\delta}^2(r-1)(r+1)r^{-2}|w|^2dr+\int_2^{\infty}(1-r^{-2})|w|^2dr\\
&\geq\delta\int_{1+\delta}^2|w/r|^2dr+1/2\int_2^{\infty}|w/r|^2dr\geq \delta\int_{1+\delta}^{\infty}|w/r|^2dr,
\end{align*}
which gives
\begin{align*}
&\|w/r\|_{L^2(1+\delta, \infty)}^2
\leq|kB|^{-1}{\delta}^{-1}\bigl|\langle F,w\rangle_\Ltr\bigr|.
\end{align*}
Thus, we obtain the following estimate for $\|w/r\|_{\Ltr}$:
\begin{align*}
\|w/r\|_{\Ltr}^2=&\|w/r\|_{L^2(1+\delta, \infty)}^2+\|w/r\|_{L^2(1,1+\delta)}^2\\
\leq&|kB|^{-1}{\delta}^{-1}\bigl|\langle F,w\rangle_\Ltr\bigr|+\int_1^{1+\delta}r^{-1}dr\|w/r^{1/2}\|_{\Lor}^2 \\
\leq&|kB|^{-1}{\delta}^{-1}\bigl|\langle F,w\rangle_\Ltr\bigr|+\delta\|w/r^{1/2}\|_{\Lor}^2,
\end{align*}
from which,  Lemma \ref{trivial w' lemma} and Lemma \ref{Appendix A1-1}, we deduce that
 \begin{align*}
\|w/r\|_{\Ltr}^2
\leq&|kB|^{-1}{\delta}^{-1}\bigl|\langle F,w\rangle_\Ltr\bigr|+\delta\|w/r^{1/2}\|_{\Lor}^2\\
\lesssim&|kB|^{-1}{\delta}^{-1}\bigl|\langle F,w\rangle_\Ltr\bigr|+{\delta} \|w/r\|_{\Ltr}\|w'\|_{\Ltr} + \delta \|w/r\|_{\Ltr}^2\\
\lesssim&|kB|^{-1}{\delta}^{-1}\bigl|\langle F,w\rangle_\Ltr\bigr|+\delta\nu^{-1/2}\bigl|\langle F,w\rangle_\Ltr\bigr|^{1/2}\|w/r\|_{\Ltr} + \delta \|w/r\|_{\Ltr}^2.
\end{align*}
Hence by taking $\delta=\nu^{1/3}|kB|^{-1/3}$,
we achieve
\begin{align}\label{lmabda-0-1}
\bigl|\langle F,w\rangle_\Ltr\bigr|\gtrsim\min\{|kB|\delta,{\delta}^{-2}\nu\}\|w/r\|_{\Ltr}^2\gtrsim \nu^{1/3}|kB|^{2/3}\|w/r\|_{\Ltr}^2.
\end{align}

\no \textbf{(3). Case of $\lambda\in(0,1)$.}

Let $r_0=\lambda^{-\f12}$ and  $\delta =\tilde{\delta}r_0$, where $0<\tilde{\delta}\ll 1$ is a small constant to be determined later. In order to control $\|w/r\|_{\Ltr((1, r_0-\delta)\cup(r_0+\delta,\infty))}$, we denote $\rho_\delta:= \rho(\frac{r-r_0}{\delta})$, $\rho_\delta':= \rho'(\frac{r-r_0}{\delta})$  and take $\Ltr$ inner product of $F$ with $\rho_\delta w$, and then we get, by using integration by parts, that
\begin{align*}
\langle F,\rho_\delta w \rangle_\Ltr
=&-\nu\int_1^{\oo} \rho_\delta w''\overline{w}dr+\nu (k^2-1/4)\int_1^{\oo} \rho_\delta |w/r|^2dr +ikB\int_1^{\oo} \rho_\delta  (r^{-2}-\lambda)|w|^2dr\\
=&\nu\int_1^{\oo}\rho_\delta |w'|^2dr + \nu\int_1^{\oo} \delta^{-1}\rho'_\delta w' \overline{w} dr\\
&+\nu (k^2-1/4)\int_1^{\oo} \rho_\delta |w/r|^2dr +ikB\int_1^{\oo} \rho_\delta  (r^{-2}-\lambda)|w|^2dr.
\end{align*}
By taking the imaginary part of above equality and noticing the sign of $\rho_\delta$, one has
\begin{align*}
&|kB|\Bigl(\int_1^{r_0-\delta} (r^{-2}-\lambda)|w|^2dr+\int_{r_0+\delta}^\infty (\lambda-r^{-2})|w|^2dr\Bigr)\\
&\leq\big|\langle F,\rho_\delta w\rangle_\Ltr\big|+ \nu \delta^{-1}\Bigl|\int_1^{\oo} \rho'_\delta w' \overline{w} dr\Bigr| + |kB| \int_{r_0-\delta}^{r_0+\delta}|r^{-2}-\lambda| |w|^2 dr .
\end{align*}
In view of the definitions of $r_0, \delta$, we have
\begin{align*}
    &\int_1^{r_{-}} (r^{-2}-\lambda)|w|^2dr\geq \int_1^{r_0-\delta}(r^{-2}-r_0^{-2})|w|^2dr\\
   & =\int_1^{r_0-\delta}  (r_0-r)(r_0+r)r^{-2}r_0^{-2}|w|^2dr\geq\delta/ r_0\int_1^{r_0-\delta}| w/r|^2dr= \tilde{\delta}\int_1^{r_0-\delta}| w/r|^2dr,
\end{align*}
and
\begin{align*}
    &\int_{r_{+}}^\infty  (\lambda -r^{-2})|w|^2dr\geq \int_{r_0+\delta}^\infty  (r_0^{-2}-r^{-2})|w|^2dr\\
    &=\int_{r_0+\delta}^\infty  (r-r_0)(r_0+r)r^{-2}r_0^{-2}|w|^2dr\geq\delta/ r_0\int_{r_0+\delta}^\infty| w/r|^2dr=\tilde{\delta}\int_{r_0+\delta}^\infty| w/r|^2dr.
\end{align*}
And thanks to the support of $\rho_\delta'$, we have
\begin{align*}
    &\nu\delta^{-1} \Bigl|\int_1^{\oo} \rho'_\delta w' \overline{w} dr \Bigr| \lesssim    \nu\delta^{-1} r_0 \Bigl|\int_1^{\oo} \rho_\delta'  w' \overline{w}/r dr \Bigr|\lesssim  \nu \tilde{\delta}^{-1} \| w'\|_\Ltr \|w/r\|_\Ltr,
\end{align*}
and
\begin{align*}
    |kB| \int_{r_0-\delta}^{r_0+\delta}|r^{-2}-\lambda| |w|^2 dr &=|kB| \int_{r_0-\delta}^{r_0+\delta} \frac{(r_0+r)|r_0-r|}{r^2 r_0^2}| |w|^2 dr \\
    &\lesssim |kB|\tilde{\delta} \|w/r\|_{L^2(r_0-\delta,r_0+\delta)}^2.
\end{align*}

By summarizing the above inequalities, we can bound $\| w/r\|_{L^2((1, r_0-\delta)\cup(r_0+\delta,\infty))}$ as follows:
\begin{equation*}\begin{split}
\label{w/r 1 to r_+ r_+ to R}
\| w/r\|_{L^2((1, r_0-\delta)\cup(r_0+\delta,\infty))}^2\lesssim &|kB|^{-1}\tilde{\delta}^{-1}\Bigl(\big|\langle F,\rho_\delta w\rangle_\Ltr\big|+\nu\tilde{\delta}^{-1}\| w'\|_\Ltr \|w/r\|_\Ltr \Bigr)\\
&+ \|w/r\|_{L^2(r_0-\delta,r_0+\delta)}^2,
\end{split}\end{equation*}
from which and Lemma \ref{Appendix A1-2}, we infer
\begin{align*}
\|w/r\|_{\Ltr}^2=&\|w/r\|_{L^2((1, r_0-\delta)\cup(r_0+\delta,\infty))}^2+\|w/r\|_{L^2(r_0-\delta,r_0+\delta)}^2\\
\lesssim&|kB|^{-1}\tilde{\delta}^{-1}\Big(\big|\langle F,\rho_\delta w\rangle_\Ltr\big|+\nu\tilde{\delta}^{-1}\|w'\|_{\Ltr}\|w/r\|_{\Ltr}\Big)+\int_{r_0-\delta}^{r_0+\delta}r^{-1}dr\|r^{-\f12}w\|_{\Lor}^2 \\
\lesssim&|kB|^{-1}\tilde{\delta}^{-1}\Big(\big|\langle F,\rho_\delta w\rangle_\Ltr\big|+\nu\tilde{\delta}^{-1}\|w'\|_{\Ltr}\|w/r\|_{\Ltr}\Big)+\tilde{\delta}\|r^{-\f12}w\|_{\Lor}^2.
\end{align*}

 \noindent Then thanks to  Lemma \ref{trivial w' lemma} and Lemma \ref{Appendix A1-1}, we deduce that
\begin{align*}
    \|w/r\|_{\Ltr}^2
&\lesssim|kB|^{-1}\tilde{\delta}^{-1}\Big(\big|\langle F,\rho_\delta w\rangle_\Ltr\big|+\nu\tilde{\delta}^{-1}\|w'\|_{\Ltr}\|w/r\|_{\Ltr}\Big)+\tilde{\delta}\|r^{-\f12}w\|_{\Lor}^2\\
&\lesssim|kB|^{-1}\tilde{\delta}^{-1}\Big(\big|\langle F,\rho_\delta w\rangle_\Ltr\big|+\nu\tilde{\delta}^{-1}\|w'\|_{\Ltr}\|w/r\|_{\Ltr}\Big)\\
&\quad + \tilde{\delta}\|w/r\|_\Ltr \bigl(\|w'\|_\Ltr+  \|w/r\|_\Ltr \bigr)\\
&\lesssim |kB|^{-1}\tilde{\delta}^{-1}\big|\langle F,\rho_\delta w\rangle_\Ltr\big| + \tilde{\delta} \|w/r\|_\Ltr^2 + \bigl(\tilde{\delta}+\nu|kB|^{-1}\tilde{\delta}^{-2}\bigr)\|w/r\|_\Ltr \bigl\|w'\|_\Ltr\\
&\lesssim |kB|^{-1}\tilde{\delta}^{-1}\big|\langle F,\rho_\delta w\rangle_\Ltr\big|+\bigl(\tilde{\delta}+\nu|kB|^{-1}\tilde{\delta}^{-2}\bigr)\nu^{-\f12} \|w/r\|_{\Ltr}  |\braket{F,w}_\Ltr|^\f12 + \tilde{\delta} \|w/r\|_\Ltr^2  .
\end{align*}

Finally by choosing $\tilde{\delta}=\nu^{1/3}|kB|^{-1/3}\ll1$ to optimize the above inequality and  absorbing the term $\tilde{\delta} \|w/r\|_\Ltr^2 $, we achieve
\begin{align*}
\|w/r\|_{\Ltr}^2
&\lesssim |kB|^{-1}\tilde{\delta}^{-1}\big|\langle F,\rho_\delta w\rangle_\Ltr\big|+\tilde{\delta}\nu^{-1/2}\bigl|\langle F,w\rangle_\Ltr\bigr|^{1/2}\| w/r\|_{\Ltr}.
\end{align*}
Applying Young inequality and using $\tilde{\delta}=\nu^{1/3}|kB|^{-1/3}\ll1$ gives
\begin{align}
\label{lmabda-1}\bigl|\langle F,w\rangle_\Ltr\bigr| + \big|\langle F,\rho_\delta w\rangle_\Ltr\big|\gtrsim \nu^{1/3}|kB|^{2/3}\|w/r\|_{\Ltr}^2.
\end{align}

Therefore, by summarizing \eqref{lmabda trival-0}, \eqref{lmabda-0-1} and \eqref{lmabda-1}, we conclude that
\begin{align}
\label{k.6-}\nu^{1/3}|kB|^{2/3}\|w/r\|_{\Ltr}^2\lesssim\bigl|\langle F,w\rangle_\Ltr\bigr| +\big|\langle F,\rho_\delta w\rangle_\Ltr\big|\mathds{1}_{[0,1]}(\lambda) \quad \textrm{for all} \ \lambda\in \mathbb{R},
\end{align}
which along with Lemma \ref{trivial w' lemma} completes the proof of \eqref{resolvent estimate-0-1}.

To check \eqref{2.4}, one only needs to notice that:
\begin{equation*}
    \big|\langle F_1,w\rangle_\Ltr \big| +\big|\langle F_1,\rho_\delta w\rangle_\Ltr\big|\mathds{1}_{[0,1]}(\lambda) \lesssim \|rF_1\|_\Ltr \|w/r\|_\Ltr,
\end{equation*}
and since $w(1)=0$ and $\frac{r_0}{\delta}=\tilde{\delta}^{-1},$ we have
\begin{align*}
    \big|\langle \pa_rF_2,w\rangle_\Ltr\big| +\big|\langle \pa_r F_2,\rho_\delta w\rangle_\Ltr\big|\mathds{1}_{[0,1]}(\lambda) &\lesssim \|F_2\|_\Ltr \bigl(\|\pa_r w\|_\Ltr + \frac{r_0}{\delta} \|w/r\|_\Ltr\bigr)\\
    &\lesssim \|F_2\|_\Ltr \bigl(\|\pa_r w\|_\Ltr +{ \nu^{-\f13}|kB|^\f13 }\|w/r\|_\Ltr\bigr).
\end{align*}
By substituting the above inequalities into \eqref{resolvent estimate-0-1} and using Young's inequality, we obtain
\begin{align*}\nu\|w'\|_{\Ltr}^2+\mu_k\|w/r\|_{\Ltr}^2\leq C\bigl(\mu_k^{-1}\|rF_1\|_\Ltr^2 +\nu^{-1}\|F_2\|_\Ltr^2\bigr)
+\frac12\bigl(\nu\|w'\|_{\Ltr}^2+\nu^{\f13}|kB|^{\f23 }\|w/r\|_\Ltr^2\bigr),
\end{align*}
which leads to \eqref{2.4}. This completes the proof of Lemma \ref{resolvent estimate-1}.
 \end{proof}

Now we prove that the pseudospectral bound $\nu^{1/3} |kB|^{2/3}$ in Lemma \ref{resolvent estimate-1} is sharp for low frequency $k$:

\begin{lemma}\label{resolvent estimate-2}
Let  $|B|>\nu$. There exist $\lambda_0\in \R$, $w_0\in D$ and $C,$ which are  independent of $\nu$ and $ B$, so that
\begin{align*}
\|r(\mathcal{T}-iB\lambda_0) w_0\|_{\Ltr}\leq C\nu^{1/3} |B|^{2/3}\|w_0/r\|_{\Ltr}.
\end{align*}
\end{lemma}
\begin{proof}
Let $r_0=\left(\frac{|B|}{\nu}\right)^\f16>1$ and $\lambda_0=r_0^{-2}$.  We construct
\begin{align*}
w_0(r)=\left\{
\begin{aligned}
&(r-r_0)^3(r_0+r_0^{-1}-r)^3,\quad r_0\leq r\leq r_0+r_0^{-1},\\
&0,\quad \textrm{else}.
\end{aligned}
\right.
\end{align*}
It's easy to check that $w_0\in D$. Due  to $r_0>1$, we have $r_0+r_0^{-1}\leq 2r_0$, from which, we get,
by  direct calculations, that
\begin{align*}
  &  \bigl\|\frac{w_0}{r}\bigr\|_\Ltr = \Bigl(\int_{r_0}^{r_0+r_0^{-1}}r^{-2}   (r-r_0)^6(r_0+r_0^{-1}-r)^6 dr\Bigr)^\f12  \approx r_0^{-\frac{15}{2}},\\
&\nu\| r(-\partial_r^2+3/4r^{-2})w_0\|_{\Ltr}\leq \nu\bigl(\|r\partial_r^2w_0\|_{\Ltr}+\bigl\| \frac{w_0}{r}\bigr\|_{\Ltr}\bigr) \lesssim \nu(r_0^4+1)r_0^{-\frac{15}{2}} \lesssim \nu r_0^4\bigl\|\frac{w_0}{r}\bigr\|_\Ltr,\end{align*}
and
\begin{align*}
|B|\|r(\lambda_0-r^{-2})w_0\|_{\Ltr}=&|B|\|r(r_0^{-2}-r^{-2})w_0\|_{\Ltr}\\
 \leq &|B|\bigl\|\frac{(r-r_0)(r+r_0)}{r_0^2}\frac{w_0}{r}\bigr\|_{\Ltr} \lesssim  |B|r_0^{-2}\bigl\|\frac{w_0}{r}\bigr\|_{\Ltr}.
\end{align*}
By summarizing the above inequalities, we arrive at
\begin{align*}
\|r(\mathcal{T}-iB\lambda_0) w_0\|_{\Ltr} \lesssim (\nu r_0^4+|B|r_0^{-2})\|w_0/r\|_{\Ltr} \lesssim \nu^{1/3} |B|^{2/3}\|w_0/r\|_{\Ltr}.
\end{align*}
This completes the proof of Lemma \ref{resolvent estimate-2}.
\end{proof}

\subsection{Space-time estimates for $\wkt$} We recall from \eqref{homogenous eq tc} and \eqref{inhomogenous eq tc} that $\wkt=\wkht+\wkit$, where
 \begin{align*}
     &\partial_t \wkht +\mathcal{T}\wkht = 0,\quad \wkht|_{t=0} =w_k(0),\quad \wkht|_{r=1,\infty}=0,\\
      &\partial_t \wkit +\mathcal{T}\wkit = f_k, \quad \wkit|_{t=0}=0, \quad  \wkit|_{r=1,\infty}=0.
\end{align*}
In this subsection, our main aim is to prove the following estimate:
\begin{proposition}\label{space-time estimates}
    For any $k\in \mathbb{Z}\backslash \{0 \}$, $|kB|\gg \nu >0$ and $T>0$, there exists constant $C>0$ independent of $\nu,k,B,T$, so that
    \begin{equation*}
   \|\wkt\|_{\LoT(\Ltr)} + \nu^{1/2} \|\partial_r \wkt\|_{\LtT(\Ltr)} + \mu_k^{1/2}\|\wkt/r\|_{\LtT(\Ltr)} \leq C\bigl( \|w_k(0)\|_{L^2} + \mu_k^{-1/2} \|rf_k\|_{\LtT(\Ltr)}\bigr),
    \end{equation*}
    where $\mu_k=\max\{\nu k^2, \nu^{1/3} |kB|^{2/3} \}$.
\end{proposition}
\begin{proof}This is a direct consequence of \eqref{L2 space time estimates} and Proposition \ref{estimate of homogeneous eqs}.
\end{proof}

\subsubsection{Space-time estimates of the inhomogeneous part $\wkit$}\label{Space-time estimates of the inhomogeneous equation inhomogenous eq}

%\begin{lemma}\label{the inhomogeneous equation with zero initial data}
%   For any $k\in \mathbb{Z}\backslash \{0 \}$, $\exists$ a constant $C>0$ independent of $\nu,k,B$, such that
%\begin{equation}\label{ineq inhomogenous equations}
 %\begin{split}
  %  &\|\wkit \|_{L^\oo L^2}+\nu^{1/2}\|\partial_r\wkit\|_{L^2L^2}+\mu_k^{1/2} \|\wkit/r\|_{L^2L^2}
  %  \leq C  \big\|\langle f_k,\wkit\rangle_\Ltr\big\|_{L^2(0,\infty)}^{1/2}.
% \end{split}
%\end{equation}
%\end{lemma}\label{estimate of inhomogeneous eqs}
%\begin{proof}This is a corollary of Lemma \ref{abstract ineq inhomogenous equations-low mode}.
%\end{proof}

In this section, we investigate the space-time estimates of the following equation:
\begin{align}\label{2.10}
    \partial_t w +\mathcal{T}w = f=f_1+\pa_r f_2, \quad w|_{t=0}=0, \quad  w|_{r=1,\infty}=0.
\end{align}
\begin{lemma}\label{abstract ineq inhomogenous equations-low mode}
Let $k\in \mathbb{Z}\backslash \{0 \}$,  $|kB|\gg \nu >0$, $T>0$ and let   $(w,f)$ solve \eqref{2.10}.
There exists constant $C>0$ independent of $\nu,k,B,T,w,f$, so that
\begin{align}\label{L2 space time estimates}
\|w \|_{\LoT(\Ltr)}+\nu^{1/2}\|\partial_rw\|_{\LtT(\Ltr)}&+\mu_k^{1/2}\|w/r\|_{\LtT(\Ltr)}\\
    &\leq  C \bigl(\mu_k^{-\f12} \|rf_1\|_{\LtT(\Ltr)} +\nu^{-\f12} \|f_2\|_{\LtT(\Ltr)}\bigr).\nonumber
\end{align}
\end{lemma}
\begin{proof}
Without loss of generality, we set $T=\oo$ and define\footnote{Note that $w|_{t=0}=0$ ensures that \(w\) can be continuously extended by zero to \((-\infty, 0)\), thereby guaranteeing that \(t\) can undergo Fourier transform over the entire space \(\mathbb{R}\).}
	\begin{align*} &W(\lambda,r):=\int_0^{\infty}w(t,r)e^{-it\lambda}dt, \quad F(\lambda,r):=\int_0^{\infty}f(t,r)e^{-it\lambda}dt,  \\
&\qquad \quad F_\star(\lambda,r):=\int_0^{\infty}f_\star(t,r)e^{-it\lambda}dt, \quad \text{for}\ \star=1,2.
	\end{align*}
Then the inhomogeneous equation can thus be transferred into
	\begin{align*} (ikB\lambda+\mathcal{T})W(kB\lambda,r)=F=\left(F_1+\pa_r F_2\right)(kB\lambda,r),
	\end{align*}
from which and  Lemma \ref{resolvent estimate-1}, we infer
\begin{equation*}
\nu \|\partial_{r}W(kB\lambda,\cdot)\|_{\Ltr}^2 +\mu_k \|W(kB\lambda,r)/r\|_{\Ltr}^2\lesssim  \mu_k^{-1} \|rF_1(kB\lambda,r)\|_{\Ltr}^2 + \nu^{-1}\|F_2(kB\lambda,\cdot)\|_{\Ltr}^2.
 \end{equation*}
By integrating the above inequality over $\R$ with respect to $\lambda$ variable, we obtain
\begin{equation*}
\nu \|\partial_{r}W\|_{L^2_\lambda(\R,\Ltr)}^2 +\mu_k \|W/r\|_{L^2_\lambda(\R,\Ltr)}^2\lesssim  \mu_k^{-1} \|rF_1\|_{L^2_\lambda(\R,\Ltr)}^2 + \nu^{-1}\|F_2\|_{L^2_\lambda(\R,\Ltr)}^2.
 \end{equation*}
Then we get, by using Plancherel's equality, that
\begin{equation}\label{2.13}
\nu\|\partial_{r} w\|_{\LtT(\Ltr)}^2 +\mu_k\|w/r\|_{\LtT(\Ltr)}^2\lesssim \mu_k^{-1} \|rf_1\|_{\LtT(\Ltr)}^2+\nu^{-1} \|f_2\|_{\LtT(\Ltr)}^2.
\end{equation}

Whereas we get,  by taking the real part of the $\Ltr$ inner product of \eqref{2.10} with $w$,
	\begin{align*}
\frac{1}{2}\frac{d}{dt}\|w\|_{\Ltr}^2\leq \bigl|\langle f,w\rangle_\Ltr\bigr| \leq \|rf_1\|_\Ltr\|w/r\|_\Ltr + \|f_2\|_\Ltr \|\pa_r w\|_\Ltr.
	\end{align*}
By integrating the above inequality over $[0,T]$ and then using \eqref{2.13}, we obtain
\begin{align*}
 \|w \|_{\LoT(\Ltr)}^2 \lesssim  \mu_k^{-1} \|rf_1\|_{L^2_T(\Ltr)}^2 +\nu^{-1} \|f_2\|_{L^2_T(\Ltr)}^2,
 \end{align*}
which together with \eqref{2.13} ensures \eqref{L2 space time estimates}. We thus complete the proof of Lemma
\ref{abstract ineq inhomogenous equations-low mode}.
\end{proof}

%according to the Plancherel's theorem, we have $L^2_rL^2_t\sim L^2_rL^2_\lambda$, that is
%	\begin{align*}
%&\|\partial_{r}w\|_{L^2L^2}\approx\big\|\|\partial_{r}W\|_{L^2(1,\infty)}\big\|_{L^2(\mathbb{R})}, \quad \|w/r\|_{L^2L^2}\approx\big\|\|w/r\|_{L^2(1,\infty)}\big\|_{L^2(\mathbb{R})}, \\
%&\big\|\langle f,w\rangle_\Ltr\big\|_{L^2(0,\infty)}\approx\big\|\langle F,W\rangle_\Ltr\big\|_{L^2(\mathbb{R})}.
%	\end{align*}
%Plugging in all estimates above, we thus deduce
%	\begin{align}
%\label{spacetime estimate lemma w L2L2}&\nu\|\partial_{r}w\|_{L^2L^2}^2+\mu_k\|w/r\|_{L^2L^2}^2	\leq  C \big\|\langle f,w\rangle_\Ltr\big\|_{L^2(0,\infty)}.
%	\end{align}

\subsubsection{Space-time estimates of the homogeneous part $\wkht$}

\begin{proposition}\label{estimate of homogeneous eqs}
For any $k \in \Z\backslash\{0\}$,  $|kB|\gg \nu >0$ and $T>0$, there exists constant $C>0$ independent of $\nu,k,B,T$, so that
    \begin{equation}
\label{estimate of homogeneous eqs-1}\begin{split}
     & \|\wkht\|_{\LoT(\Ltr)} + \nu^{1/2} \|\partial_r \wkht\|_{\LtT(\Ltr)} + \mu_k^{1/2}\|\wkht/r\|_{\LtT(\Ltr)} \leq C \|w_k(0)\|_{\Ltr}.
    \end{split}
    \end{equation}
\end{proposition}
\begin{proof}We decompose $\wkht=e^{-ikBt/r^2}w^{(1)}+w^{(2)}$ with $w^{(1)}$ and $w^{(2)}$ satisfying
\begin{align}
\label{w-1} &\partial_tw^{(1)}-\nu\bigl(\partial_r^2- (k^2+\Theta^2)/r^2\bigr)w^{(1)}+\nu|2kBt/r^3|^2w^{(1)}=0,\\
\label{w-2}&\partial_tw^{(2)}+\mathcal{T}w^{(2)}= \nu\Big\{i\partial_r(e^{-ikBt/r^2}4kBt/r^3w^{(1)})\\
&\qquad \qquad   +e^{-ikBt/r^2}\bigl(i6kBt/r^4w^{(1)}+8|kBt/r^3|^2w^{(1)}+(\Theta^2+1/4)/r^2w^{(1)}\bigr)\Big\}, \nonumber \\
& w^{(1)}|_{t=0}=w_k(0), \quad w^{(2)}|_{t=0}=0, \quad w^{(1)}|_{r=1,\oo}=w^{(2)}|_{r=1,\oo}=0.
\end{align}
 Here $\Theta>0$ is a fixed constant which is sufficiently large below,  we may just take $\Theta\geq 10^9$ for instance.

 In view of \eqref{w-2} and  and $\mu_k^{-1}\leq \min\{\nu^{-1}k^{-2}, \nu^{-\f13}|kB|^{-\f23}\}$, we first get
 by applying Lemma \ref{abstract ineq inhomogenous equations-low mode}, that
 \begin{equation}\label{S2eq2}
     \begin{split}
&\|w^{(2)} \|_{\LoT(\Ltr)}+\nu^{1/2}\|\partial_rw^{(2)}\|_{\LtT(\Ltr)}+\mu_k^{1/2}\|w^{(2)}/r\|_{\LtT(\Ltr)} \\
& \lesssim \nu^{1/2}\|kBt/r^3w^{(1)}\|_{\LtT(\Ltr)}
+\mu_k^{-\f12}\nu \Bigl( \|kBt/r^3w^{(1)}\|_{\LtT(\Ltr)}\\
&\qquad\qquad \qquad\qquad\qquad\qquad\qquad\qquad+ \|(kBt)^2/r^5w^{(1)}\|_{\LtT(\Ltr)} +  \|w^{(1)}/r\|_{\LtT(\Ltr)} \Bigr)\\
& \lesssim \nu^{1/2}\|kBt/r^3w^{(1)}\|_{\LtT(\Ltr)} + \mu_k^{\f12}   \|w^{(1)}/r\|_{\LtT(\Ltr)}
+\mu_k^{-\f12}\nu  \|(kBt)^2/r^5w^{(1)}\|_{\LtT(\Ltr)} \\
& \lesssim \nu^{1/2}\|kBt/r^3w^{(1)}\|_{\LtT(\Ltr)} + \mu_k^{\f12}   \|w^{(1)}/r\|_{\LtT(\Ltr)}
+ \nu^\f56 |kB|^\f53\| t^2/r^5w^{(1)}\|_{\LtT(\Ltr)} .
     \end{split}
 \end{equation}

 Then the proof of Proposition \ref{estimate of homogeneous eqs} relies
Proposition \ref{estimate of homogeneous eqs-main part} and Lemma \ref{estimate of homogeneous eqs-main part-additional}, the proof of which will be postponed  below.

\begin{proposition}\label{estimate of homogeneous eqs-main part}
Let $k\in \Z\backslash\{0\}$,  $|kB|\gg \nu >0$ and $T>0$ and let $w^{(1)}$ solve  \eqref{w-1}. Then there exists constant $C>0$ independent of $\nu,k,B,T$, so that
\begin{equation}
    \label{eqs-0}
    \begin{split}
\|w^{(1)}\|_{\LoT(\Ltr)}^2 + \nu \|\partial_r w^{(1)}\|_{\LtT(\Ltr)}^2&+ \mu_k \|w^{(1)}/r\|_{\LtT(\Ltr)}^2\\
&+ \nu\|kBt/r^3w^{(1)}\|_{\LtT(\Ltr)}^2 \leq
 C\|w_k(0)\|_{ \Ltr}^2.
\end{split}
\end{equation}
\end{proposition}

\begin{lemma}\label{estimate of homogeneous eqs-main part-additional}
Let $k\in \Z\backslash\{0\}$, $|kB|\gg \nu >0$, $T>0$ and let $w^{(1)}$ solve  \eqref{w-1}. Then there exists constant $C>0$ independent of $\nu,k,B,T$, so that
\begin{equation}
    \label{eqs-additional}\begin{split}
&\nu^\f56 |kB|^\f53\| t^2/r^5w^{(1)}\|_{\LtT(\Ltr)} =\kappa_k^\f52\| t^2/r^5w^{(1)}\|_{\LtT(\Ltr)}
\leq C\|w_k(0)\|_{\Ltr}.
\end{split}
\end{equation}
\end{lemma}

We admit Proposition \ref{estimate of homogeneous eqs-main part} and Lemma \ref{estimate of homogeneous eqs-main part-additional}
for the time being and continue our proof of Proposition \ref{estimate of homogeneous eqs}. Indeed by substituting
\eqref{eqs-0} and \eqref{eqs-additional} into \eqref{S2eq2}, we achieve
\begin{align*}
&\|w^{(2)} \|_{\LoT(\Ltr)}+\nu^{1/2}\|\partial_rw^{(2)}\|_{\LtT(\Ltr)}+\mu_k^{1/2}\|w^{(2)}/r\|_{\LtT(\Ltr)} \lesssim\|w_k(0)\|_{\Ltr}^2,
\end{align*}
which together with \ref{estimate of homogeneous eqs-main part} ensures (\ref{estimate of homogeneous eqs-1}). We thus complete
the proof of Proposition \ref{estimate of homogeneous eqs}.
\end{proof}

 We remark that the proof of Proposition \ref{estimate of homogeneous eqs} has been reduced to estimating  $w^{(1)}$.
In view of the  $w^{(1)}$ equation \eqref{w-1}, roughly speaking, the enhanced dissipation near TC flow is reflected in the damping term $\nu |kBt/r^3|^2$, which provides various space-time estimates for $w^{(1)}$. This is the key point in proving  Proposition \ref{estimate of homogeneous eqs-main part}, which is the most crucial estimate in proving Theorem \ref{main them-1}. Moreover, we shall observe from the proof of Proposition \ref{estimate of homogeneous eqs-main part} below  that the estimate of \( \mu_k^{1/2}\|w^{(1)}/r\|_{\LtT(\Ltr)} \) is the most critical ingredient.

Let us turn to the proof of Proposition \ref{estimate of homogeneous eqs-main part} and Lemma \ref{estimate of homogeneous eqs-main part-additional}. We first present the proof of Proposition \ref{estimate of homogeneous eqs-main part}.

\begin{proof}[Proof of Proposition \ref{estimate of homogeneous eqs-main part}]  We shall  divide the proof  into the following four steps:

\no {\bf Step 1. Fundamental energy estimates.}

 We  get,  by first taking $\Ltr$ inner product of  \eqref{w-1} with \( w^{(1)} \) and then by integrating the resulting inequality over $[0,T]$, that
\begin{equation}
    \label{eqs-trivial}\begin{split}
    \|w^{(1)}(T) \|_{ \Ltr}^2+2\nu\|\partial_rw^{(1)}\|_{\LtT(\Ltr)}^2&+2\nu(k^2+\Theta^2)\|w^{(1)}/r\|_{\LtT(\Ltr)}^2\\
    &+8\nu\|kBt/r^3w^{(1)}\|_{\LtT(\Ltr)}^2
\leq\|w_k(0)\|_{ \Ltr}^2.
\end{split}
\end{equation}
Notice that $\mu_k=\max\{\nu k^2,\kappa_k \},$ it remains to deal with the
 estimate of $\nu^\f13 |kB|^\f23 \|w^{(1)}/r\|_\Ltr^2,$ which we present below.

 \no{\bf Step 2. Dyadic decomposition.}

 We first perform a dyadic decomposition of the region \([1, +\infty)\). Let
\begin{align*}
    \phi(r)=\begin{cases}
    (r-3/4)^3\bigl(6\cdot 12^5r^2-123\cdot 12^4r+631\cdot12^3\bigr), &\quad r\in[3/4,5/6], \\
    1, &\quad r\in[5/6,11/6], \\
    (23/12-r)^3\bigl(6\cdot 12^5(8/3-r)^2-123\cdot 12^4(8/3-r)+631\cdot12^3\bigr), &\quad r\in[11/6,23/12], \\
    0, &\quad \text{else}.
    \end{cases}
\end{align*}
It's easy to check that $\phi \in C^2(\R_+)$, \(\operatorname{supp}_r \phi \subset [3/4,23/12]\subset [2/3, 2]\) and for $r\in[3/4,5/6]$,
\begin{align*}
    |\phi'(r)|&\leq  3(r-3/4)^2\big|6\cdot 12^5r^2-123\cdot 12^4r+631\cdot12^3\big|+(r-3/4)^3\big|12^6r-123\cdot12^4\big| \\
    &\leq 3\cdot 12^{-2}\big(6\cdot 12^5\cdot(3/4)^2-123\cdot 12^4\cdot 3/4+631\cdot12^3\big) + 12^{-3}\big|12^6\cdot 3/4-123\cdot12^4\big| \\
    &\leq 3\cdot 12^{-2} \cdot 10\cdot 12^3 + 12^{-3}\cdot 15 \cdot 12^4 = 540
\end{align*}
and
\begin{align*}
    |\phi''(r)| &\leq 6(r-3/4)\big|6\cdot 12^5r^2-123\cdot 12^4r+631\cdot12^3\big| + 6(r-3/4)^2\big|12^6r-123\cdot12^4\big| \\
    &\quad + (r-3/4)^3 12^6\\
    &\leq 1/2\cdot 10\cdot 12^3 + 6\cdot 12^{-2}\cdot 15 \cdot 12^4 + 12^{-3}\cdot12^6 = 162\cdot 12^2.
\end{align*}
Then it follows from  the symmetric property of $\phi(r)$ that $|\phi'(r)| \leq 540$ and $|\phi''(r)|\leq 162\cdot 12^2$ for any $r\in [1,\infty)$.

Let
\begin{align*}
    \phi_j(r) := \phi\left(\frac{r}{2^j}\right) \andf \chi_j(r) := \frac{\phi_j(r)}{\sum_{j\geq 0 } \phi_j(r)},\quad \forall j \in \mathbb{N},
\end{align*}
then it can be verified that \(\sum_{j\geq 0 } \phi_j(r) \geq 1\) and
\begin{equation*}
\begin{cases}
    1\geq\chi_j(r)\geq0,   \quad \sum_{j\geq 0 }\chi_j(r) = 1, \quad 1\geq \sum_{j\geq 0 } \chi_j^2(r) \geq 1/2, \ \forall r\geq 1, \quad\chi_0(1)=1,\\
    \operatorname{supp}_r\chi_j \subset [2^{j+1}/3,2^{j+1}],\quad \chi_j\chi_\ell=0\quad\mbox{if}\quad |j-\ell|\geq 2.
\end{cases}
\end{equation*}
Furthermore, we have
\begin{align*}
|\partial_r \chi_j(r)| &\leq \frac{|\partial_r\phi_j(r)|}{\sum_{j\geq 0 }\phi_j(r)} + \frac{\phi_j(r) \bigl|\partial_r\sum_{j\geq 0 }\phi_j(r)\bigr|}{\bigl(\sum_{j\geq 0 } \phi_j(r)\bigr)^2}\leq  7\cdot 540\cdot 2^{-j} \leq  4\cdot 10^6 \cdot 2^{-j},\\
 |\partial_r^2 \chi_j(r)|&\leq \frac{|\partial_r^2\phi_j(r)|}{\sum_{j\geq 0 } \phi_j(r)}+ 2\frac{|\partial_r\phi_j(r)| \bigl|\partial_r\sum_{j\geq 0 } \phi_j(r)\bigr|}{\bigl(\sum_{j\geq 0 } \phi_j(r)\bigr)^2}+2\frac{\phi_j(r) \bigl|\partial_r^2\sum_{j\geq 0 } \phi_j(r)\bigr|}{\bigl(\sum_{j\geq 0 } \phi_j(r)\bigr)^2}\\
 &\leq 2^{-2j}(162\cdot 12^2 + 2\cdot 540\cdot6\cdot 540 + 2\cdot 6 \cdot 162\cdot 12^2) \leq 4\cdot 10^6 \cdot 2^{-2j}.
\end{align*}
Here we used the facts: \(\sum_{j\geq 0 } \phi_j(r) \geq 1\) and $\phi_j\phi_\ell=0$ if $ |j-\ell|\geq 2$  to deduce that for $r \in \operatorname{supp}_r \chi_j,$  $ k=0,1,2$,
\begin{align*}
\bigl|\partial_r^k\sum_{j\geq 0 } \phi_j(r)\bigr| \leq \sum_{|\ell-j|\leq 1 , \ell\geq 0} 2^{-k\ell} \sup_r |\partial_r^k \phi(r)| \leq 6\cdot 2^{-kj}\sup_r |\partial_r^k \phi(r)|.
\end{align*}
Below, we shall denote \((\chi'')_j\) and \((\chi')_j\) by \(\chi_j''\) and \(\chi_j'\), respectively. As a result, it comes out
\begin{align*}
|\chi_j''|,\quad |\chi_j'| \leq 4 \cdot 10^6.
\end{align*}

By virtue of the above dyadic decomposition, we can give the following equivalent definition of \(\|r^\alpha w^{(1)}/r\|_\Ltr\) by
\[
\|r^\alpha w^{(1)}\|_{\Ltr}^2 \approx \sum_{j \geq 0} 2^{2\alpha j} \|\eta_j\|_{\Ltr}^2,
\]
where we set \(\eta_j := \chi_j w^{(1)}\), which  satisfies
\begin{align*}
\left\{
    \begin{aligned}
&\partial_t\eta_j-\nu\bigl(\partial_r^2- (k^2+\Theta^2)/r^2\bigr)\eta_j+\nu|2kBt/r^3|^2\eta_j=-\nu (2^{-2j}\chi_j''w^{(1)}+2^{1-j}\chi_j'\partial_r w^{(1)}) \\
&\eta_j|_{t=0}=\chi_jw_k(0),\ \eta_j|_{r=\max\{1,2^{j+1}/3\}}=\eta_j|_{2^{j+1}}=0, \  (t,r) \in \R_+\times [\max\{1,2^{j+1}/3\}, 2^{j+1}].
\end{aligned}
\right.
\end{align*}

Noticing that  $w^{(1)} = \sum\limits_{\ell\geq 0} \eta_\ell(r)$ for $r\geq 1$ and $\chi_j\chi_\ell=0$ if $ |j-\ell|\geq 2$, we also have
$$\chi_j''w^{(1)} = \sum\limits_{|\ell-j|\leq1, \ell\geq0}\chi_j''\eta_\ell, \qquad \chi_j'\partial_r w^{(1)}= \sum\limits_{|\ell-j|\leq1, \ell\geq0} \chi_j'\partial_r \eta_\ell .$$
To control $\nu^\f13|kB|^\f23 \|w^{(1)}/r\|_\Ltr^2$, we introduce the following energy functionals:
$$ \|w^{(1)}\|_{Y_t}^2 := \f12 \sum_{j\geq0}  \|\eta_j\|_{L^\oo_t(\Ltr)}^2
 +\sum_{j\geq0}\|\eta_j\|_{X_t}^2\with \kappa_k=\nu^{\f13}|kB|^{\f23} \andf $$
$$\|\eta_j\|_{X_t}^2:= \f12 \|\eta_j(t)\|_{\Ltr}^2+\nu\|\partial_r \eta_j\|_{L^2_t(\Ltr)}^2 + \nu(k^2+\Theta^2)\|r^{-1}\eta_j\|_{L^2_t(\Ltr)}^2+ 4\kappa_k^3 \|\tau/r^3\eta_j\|_{L^2_t(\Ltr)}^2.$$

%By virtue of \(1\geq \sum_{j\geq 0 } \chi_j^2 \geq 1/2\), we notice that
%\begin{align*}
 %   1/2 |w^{(1)}|^2 \leq \sum\limits_{j\geq 0} |\eta_j|^2 \leq |w^{(1)}|^2,
%\end{align*}
%and this will be frequently used in our proof.

\no{\bf Step 3. The control \(\kappa_k\|w^{(1)}/r\|_{L^2_t(\Ltr)}^2\) by \(\|w^{(1)}\|_{Y_t}^2\).}

First, we  divide the domain into the following two parts:
\begin{enumerate}
    \item \(D_1:= \{(\tau,r) \in (0,t) \times (1,+\infty) : r^2 \leq \kappa_k \tau \}\);
    \item \(D_2 := \{(\tau,r) \in (0,t) \times (1,+\infty) : 0 \leq \kappa_k\tau \leq r^2 \}\).
\end{enumerate}
Here $D_1$ is the region where enhanced dissipation occurs. In fact, we have
\begin{equation}\label{step3 begin}
\begin{split}
\kappa_k\|w^{(1)}/r\|_{L^2_t(\Ltr)}^2 =  \kappa_k\Bigl(\iint_{(\tau,r)\in D_1}+\iint_{(\tau,r)\in D_2}\Bigr) |w^{(1)}/r|^2 dr d\tau=:\mathcal{I}_1+\mathcal{I}_2.
\end{split}
\end{equation}
For $\mathcal{I}_1$, since $\sum\limits_{j\geq 0} \chi_j^2 \geq 1/2$, one has
\begin{equation}
    \mathcal{I}_1 \leq  \kappa_k^3\|\tau/r^3w^{(1)}\|_{L^2_t(\Ltr)}^2 \leq 2\kappa_k^3 \sum_{j\geq 0} \|\tau/r^3\eta_j\|_{L^2_t(\Ltr)}^2 \leq 1/2 \|w^{(1)}\|_{Y_t}^2.
\end{equation}
And for $\mathcal{I}_2$, thanks to $r^{-1}\chi_j \leq \frac{3}{2}\cdot 2^{-j}$, one has
\begin{equation}\label{Step 3 end}
\begin{split}
    \mathcal{I}_1 &\leq 2 \kappa_k \sum_{j\geq 0} \iint_{(\tau,r)\in D_2}|r^{-1}\eta_j|^2 drd\tau\\
    &\leq 2 \kappa_k \sum_{j\geq 0} \int_0^{\min\left\{t,2^{2j+2}/\kappa_k\right\}} \int_1^\oo 9/2^{2j+2}|\eta_j|^2 drd\tau\\
    &\leq 18 \sum_{j\geq 0}   \|\eta_j\|_{L^\oo_t(\Ltr)}^2 \leq 36 \|w^{(1)}\|_{Y_t}^2.
\end{split}
\end{equation}

By summarizing (\ref{step3 begin}-\ref{Step 3 end}), we obtain
\begin{equation}\label{step 3 final}
\kappa_k\|w^{(1)}/r\|_{L^2_t(\Ltr)}^2\lesssim \|w^{(1)}\|_{Y_t}^2.
\end{equation}

\newpage

 \no{\bf Step 4. The estimation of \(\|w^{(1)}\|_{Y_t}\).}

By taking the \(\Ltr\) inner product of \(\eta_j\) with the equation of \(\eta_j\), and using integration by parts, we find
\begin{equation*}
    \begin{split}
\frac{d}{dt}\|\eta_j\|_{X_t}^2 &\leq  \nu\sum_{|\ell-j|\leq1, \ell\geq0}\bigl(2^{-2j}\|r^2\chi_j''\|_{\Lor}\|r^{-1}\eta_\ell\|_{\Ltr}+2^{1-j}\|r\chi_j'\|_{\Lor}\|\partial_r \eta_\ell\|_{\Ltr}\bigr) \|r^{-1}\eta_j\|_{\Ltr},
    \end{split}
\end{equation*}
which together with the facts: $\operatorname{supp}_r\chi_j \subset [2^{j+1}/3,2^{j+1}]$ and \(\|\chi_j''\|_{\Lor}, \|\chi_j'\|_{\Lor} \leq 4 \cdot 10^6\), ensures that
\begin{equation*}
    \begin{split}
    \frac{d}{dt}\|\eta_j\|_{X_t}^2 &\leq 16\cdot 10^6\nu \sum_{|\ell-j|\leq1, \ell\geq0}(\|r^{-1}\eta_\ell\|_{\Ltr}+\|\partial_r \eta_\ell\|_{\Ltr}) \|r^{-1}\eta_j\|_{\Ltr}.
    \end{split}
\end{equation*}
By integrating the above inequality over \([0,t]\), and then using the definition of \(X_t\), we get
\begin{equation*}
    \begin{split}
\frac{1}{2}\|\eta_j\|_{L^\oo_t(\Ltr)}^2+\|\eta_j\|_{X_t}^2 &\leq 2\times \Bigl(1/2\|\eta_j(0)\|_{\Ltr}^2+ 16\cdot 10^6(1/\Theta+1/\Theta^2)  \|\eta_j\|_{X_t}\sum_{|\ell-j|\leq1, \ell\geq0} \|\eta_\ell\|_{X_t} \Bigr).\\
    \end{split}
\end{equation*}
By summing over \(j \geq 0\) and using  Cauchy inequality, we achieve
\begin{equation*}
\begin{split}
    \|w^{(1)}\|_{Y_t}^2 &\leq \sum_{j\geq0}\|\eta_j(0)\|_{\Ltr}^2 + 16\cdot 10^6(1/\Theta+1/\Theta^2) \Bigl( \|w^{(1)}\|_{Y_t}^2 + 3\sum_{j\geq0}\sum_{|\ell-j|\leq1, \ell\geq0} \|\eta_\ell\|_{X_t}^2\Bigr)\\
    &\leq \sum_{j\geq0}\|\eta_j(0)\|_{\Ltr}^2  +16\cdot10^7(1/\Theta+1/\Theta^2) \|w^{(1)}\|_{Y_t}^2.
\end{split}
\end{equation*}
Due to $\Theta\geq10^9$, we obtain
\begin{equation*}
\begin{split}
    \|w^{(1)}\|_{Y_t}^2 \lesssim \|w^{(1)}(0)\|_{\Ltr}^2= \|w_k(0)\|_{\Ltr}^2.
\end{split}
\end{equation*}
By inserting the above estimate into \eqref{step 3 final}, we achieve
\begin{equation*} 
\kappa_k\|w^{(1)}/r\|_{L^2_t(\Ltr)}^2 \lesssim \|w_k(0)\|_{\Ltr}^2,
\end{equation*}
which together with \eqref{eqs-trivial} ensures \eqref{eqs-0}. We thus  complete the proof of Proposition \ref{estimate of homogeneous eqs}.
\end{proof}

Next we present the proof of Lemma \ref{estimate of homogeneous eqs-main part-additional}.

\begin{proof}[Proof of Lemma \ref{estimate of homogeneous eqs-main part-additional}]
  We set $\tilde{w}:=\kappa_kt/r^2 w^{(1)}$, then in view of \eqref{w-1}, $\tilde{w}$ satisfies
\begin{align}\label{tilde{w}}
\left\{
\begin{aligned}
&\partial_t\tilde{w}-\nu\bigl(\partial_r^2- (k^2+\Theta^2)/r^2\bigr)\tilde{w}+\nu|2kBt/r^3|^2\tilde{w}= 4\nu/r\partial_r\tilde{w} +2\nu/r^2\tilde{w}+\tilde{w}/t,\\
& \tilde{w}|_{t=0}=0, \quad \tilde{w}|_{r=1,\oo}=0.
\end{aligned}
\right.
\end{align}
By taking $\Ltr$ inner product of  the above equation with \( \tilde{w} \) and using integration by parts, we obtain
\begin{equation}\label{S2eq4}
    \begin{split}
  & \f12\frac{d}{dt}\|\tilde{w} \|_{\Ltr}^2+\nu\|\partial_r\tilde{w}\|_{\Ltr}^2+\nu(k^2+\Theta^2)\|\tilde{w}/r\|_{\Ltr}^2+4\nu\|kBt/r^3\tilde{w}\|_{\Ltr}^2\\
&\leq\big|\langle4\nu/r\partial_r\tilde{w} +2\nu/r^2\tilde{w},\tilde{w}\rangle_\Ltr\big|+\big|\langle\tilde{w}/t,\tilde{w}\rangle_\Ltr\big| \\
&\leq 4\nu\|\partial_r\tilde{w}\|_{\Ltr}\|\tilde{w}/r\|_{\Ltr}+2\nu\|\tilde{w}/r\|_{\Ltr}^2+\big|\langle\tilde{w}/t,\tilde{w}\rangle_\Ltr\big|.
\end{split}
\end{equation}
Since \( \Theta \) is sufficiently large,   \( 4\nu\|\partial_r\tilde{w}\|_{\Ltr}\|\tilde{w}/r\|_{\Ltr}+2\nu\|\tilde{w}/r\|_{\Ltr}^2 \) can be absorbed by the left-hand side.

Whereas for \( \big|\langle\tilde{w}/t,\tilde{w}\rangle_\Ltr\big| \), we have
\begin{align*}
\big|\langle\tilde{w}/t,\tilde{w}\rangle_\Ltr\big|\leq& \|\kappa^\f32 t/r^3w^{(1)}\|_\Ltr \|\kappa^\f12 /rw^{(1)}\|_\Ltr  .
\end{align*}
Notice that  $\nu |kB|^2 = \kappa_k^3,$ we get, by
by integrating \eqref{S2eq4} over $[0,T]$ and using \eqref{eqs-0}, that
\begin{align*}
  &\|\tilde{w}(T) \|_{\Ltr}^2+\nu\|\partial_r\tilde{w}\|_{\LtT(\Ltr)}^2+\nu(k^2+\Theta^2)\|\tilde{w}/r\|_{\LtT(\Ltr)}^2+4\nu\|kBt/r^3\tilde{w}\|_{\LtT(\Ltr)}^2 \\
  &\lesssim \|\kappa_k^\f32 t/r^3w^{(1)}\|_{\LtT(\Ltr)} \|\kappa_k^\f12 /rw^{(1)}\|_{\LtT(\Ltr)}
  \lesssim \|w_k(0)\|_\Ltr^2.
\end{align*}
Observing that $\nu\|kBt/r^3\tilde{w}\|_{\LtT(\Ltr)}^2=\kappa_k^5\| t^2/r^5w^{(1)}\|_{\LtT(\Ltr)}^2,$
we complete the proof of Lemma \ref{estimate of homogeneous eqs-main part-additional}.
\end{proof}

\section{Proof of Theorems \ref{main them-1} and \ref{main them-2}.}

We first present the proof of Theorem \ref{main them-1}. Indeed
the proof of Theorem \ref{main them-1} follows the same lines as that of Lemma \ref{estimate of homogeneous eqs-main part-additional}.

\begin{proof}[\bf{Proof of Theorem \ref{main them-1}}]
 We introduce
    \begin{equation}
  \label{E_k} 
  \begin{split}
  &E_k(\eta) := \|\eta\|_{\LoT(\Ltr)} + \nu^{1/2} \|\partial_r \eta\|_{\LtT(\Ltr)} + \mu_k^{1/2}\|\eta/r\|_{\LtT(\Ltr)}\andf\\
   &\Lambda := 1+\kappa_kt/r^2, \quad \eta_q := \Lambda^q \wkht\quad \mbox{for}\quad q\in \mathbb{N}.
   \end{split}
    \end{equation}
In view of \eqref{homogenous eq tc}, we find
\begin{align*}
\left\{
\begin{aligned}
        & \partial_t \eta_{q} +\mathcal{T}\eta_{q} = \pa_t\Lambda^q \wkht -\nu \pa_r^2 \Lambda^q \wkht -2\nu \pa_r \Lambda^q  \pa_r\wkht\\
        &\qquad  \qquad \quad=\pa_t \Lambda^q \wkht - \nu \frac{\pa_r^2\Lambda^q}{\Lambda^q}\eta_q + 2\nu \Big(\frac{\pa_r\Lambda^q}{\Lambda^q}\Big)^2 \eta_q -2\nu \frac{\pa_r\Lambda^q}{\Lambda^q}\pa_r \eta_q \\
         &\qquad  \qquad \quad = q\pa_t \Lambda \eta_{q-1}+ \Big(  (q^2+q)\nu \Big| \frac{\pa_r \Lambda}{\Lambda} \Big|^2 - q \nu \frac{\pa_r^2 \Lambda}{\Lambda}\Big)  \eta_q  -2q \nu \frac{\pa_r \Lambda}{\Lambda}\pa_r \eta_q,\\
         &\eta_{q}(t=0) =w_k(0), \quad  \eta_{q}|_{r=1,\infty}=0.
\end{aligned}
\right.
\end{align*}
 Observing that $|\pa_t\Lambda| \leq \frac{\kappa_k}{r^2}$, $\Big|\frac{\pa_r^2\Lambda}{\Lambda}\Big| \lesssim r^{-2}$ and $\Big|\frac{\pa_r\Lambda}{\Lambda}\Big| \lesssim r^{-1}$,  we get, by applying Proposition \ref{space-time estimates}, that
\begin{equation}\label{3.2,a}
     E_k(\eta_{0}) =  E_k(\eta) \lesssim \|w_k(0)\|_\Ltr,
\end{equation}
and for $q\geq 1$
\begin{equation*}
\begin{split}
   E_k(\eta_{q})  &\leq C\|w_k(0)\|_\Ltr+ C \mu_k^{-1/2}\Big(q  \kappa_k\|\eta_{q-1}/r\|_{\LtT(\Ltr)}\\
   &\qquad\qquad\qquad\qquad\qquad\qquad+  (q^2+q)\nu \|\eta_{q}/r\|_{\LtT(\Ltr)} + q\nu \|\partial_r \eta_{q}\|_{\LtT(\Ltr)}\Big)\\
    & \leq C\|w_k(0)\|_\Ltr+  C q\kappa_k/\mu_k E_k(\eta_{q-1}) + C\bigl( (q^2+q)\nu/\mu_k+ q(\nu/\mu_k)^{1/2} \bigr) E_k(\eta_{q}).
    \end{split}
\end{equation*}
Under the assumption of Theorem \ref{main them-1} that $\frac{\nu}{|kB|} \ll (1+q)^{-3}$, we deduce from  the definition of $\kappa_k$ and $\mu_k$, that
$$\frac{\kappa_k}{\mu_k} \leq 1, \quad \frac{\nu}{\mu_k}\leq \frac{\nu}{\nu^\f13 |kB|^\f23}\ll (1+q)^{-2}.$$
So that the term $C\bigl( (q^2+q)\nu/\mu_k+ q(\nu/\mu_k)^{1/2} \bigr) E_k(\eta_{q})$ on the right-hand side can be absorbed, we thus obtain
\begin{equation}\label{3.3,a}
    E_k(\eta_{q}) \lesssim \|w_k(0)\|_\Ltr+ q E_k(\eta_{q-1}).
\end{equation}

By virtue of  \eqref{3.2,a} \eqref{3.3,a}, we get, by using an iteration argument, that
\begin{equation*}
         E_k(\eta_{q}) \lesssim_q  \|w_k(0)\|_\Ltr,
\end{equation*}
which completes the proof of Theorem \ref{main them-1}.
\end{proof}

Next we present the proof of Theorem \ref{main them-2}.

\begin{proof}[\bf{Proof of Theorem \ref{main them-2}}]
    We begin with a simpler parabolic equation:
\begin{equation}\label{S3eq4}
    \begin{cases}
         \partial_t \tilde{\eta}+\mathcal{T}\tilde{\eta} = 0, \quad (t,r)\in \R_+ \times [1,1+\delta],\\
         \tilde{\eta}(t=0) =0, \quad  \tilde{\eta}|_{r=1}=0,\quad \tilde{\eta}|_{r=1+\delta}=g(t) ,
    \end{cases}
\end{equation}
where $\delta\ll1$ will be determined  later on and $g(t)$ is a smooth function satisfying
\begin{equation*}
    g(t)=0, \quad \text{for}\ 0\leq t \leq 1 \ \text{or} \ t\geq 2.
\end{equation*}
It's easy to construct a  smooth function $f(t,r)$ which satisfies
$$f(0,r)=0,\quad f(t,1) = 0,\quad f(t,1+\delta)=g(t), $$
$$\|\nabla_{t,r}^\al f(t,r)\|_{L^\oo_r[1,1+\delta]}\lesssim \mathds{1}_{1\leq t \leq2}, \quad \forall |\al|\leq 6.$$

Let $\eta:=\tilde{\eta}-f(t,r)$ and $G(t,r):= \partial_t f+ \mathcal{T}f,$ we observe from \eqref{S3eq4} that
\begin{equation}\label{appendix a equation of eta}
    \begin{cases}
         \partial_t \eta +\mathcal{T}\eta = -G(t,r), \quad (t,r)\in \R_+ \times [1,1+\delta],\\
         \eta(t=0) =0, \quad  \eta|_{r=1}=0,\quad \eta|_{r=1+\delta}=0 .
    \end{cases}
\end{equation}
By taking the real part of the $L^2$ inner product of  (\ref{appendix a equation of eta}) with $\eta$, we obtain (in what follows, we will abbreviate $L^p((1,1+\delta),dr), H^k((1,1+\delta),dr) $ to $L^p, H^k$ respectively.)
\begin{equation}\label{a.11}
    \f12\frac{d}{dt}\|\eta(t)\|_{L^2}^2 + \nu \|\partial_r \eta\|_{L^2}^2 + \nu(k^2-1/4)\|\eta/r\|_{L^2}^2 \leq \|G\|_{L^2}\|\eta\|_{L^2}.
\end{equation}
It follows from Poincar\'e inequality on $[1,1+\delta]$ that
\begin{equation}\label{appendix a Poincare ineq}
    \delta^{-1} \|\eta\|_{L^2} \leq C_P \|\partial_r \eta\|_{L^2}, \quad \text{for any}  \ \eta \in H^1 \  \text{with} \  \eta(1)=0 .
\end{equation}
By inserting (\ref{appendix a Poincare ineq}) into (\ref{a.11}), we find
\begin{equation*}
    \begin{split}
\frac{1}{2}\frac{d}{dt}\|\eta(t)\|_{L^2}^2 +\frac{\nu}{2}\|\partial_r \eta\|_{L^2}^2+\frac{\nu}{2C_P^2\delta^2} \| \eta\|_{L^2}^2 \leq \frac{C_P^2\delta^2}{\nu}\|G\|_{L^2}^2 + \frac{\nu}{4C_P^2\delta^2} \| \eta\|_{L^2}^2,
    \end{split}
\end{equation*}
which is equivalent to
\begin{equation*}
    \begin{split}
\frac{d}{dt}\|\eta(t)\|_{L^2}^2+\nu \|\partial_r \eta\|_{L^2}^2 +\frac{\nu}{2C_P^2\delta^2} \| \eta\|_{L^2}^2 \leq \frac{2C_P^2\delta^2}{\nu}\|G\|_{L^2}^2.
    \end{split}
\end{equation*}
By integrating the above inequality over $[0,T]$, we achieve
\begin{equation}\label{appendix a time decay of eta}
    \Big\|e^{\frac{\nu t}{4C_P^2\delta^2}}\eta\Big\|_{\LoT (L^2)}^2 +   \Big\|e^{\frac{\nu t}{4C_P^2\delta^2}}\partial_r \eta\Big\|_{\LtT(L^2)}^2 \lesssim_{\nu,\delta} 1.
\end{equation}

On the other hand, by the equation \eqref{appendix a equation of eta} and boundary condition of $\eta$, we find that
 $\partial_t \eta$ and $\partial_r^2 \eta$ satisfies respectively 
\begin{equation*}
    \begin{cases}
         \partial_t \partial_t \eta +\mathcal{T}\partial_t \eta = -\partial_t G(t,r), \\
         \partial_t \eta(t=0) =0, \quad  \partial_t \eta|_{r=1}=0,\quad \partial_t \eta|_{r=1+\delta}=0 ,
    \end{cases}
\end{equation*}
and
\begin{equation*}
    \begin{cases}
         \partial_t \partial_r^2\eta +\mathcal{T}\partial_r^2 \eta = -\partial_r^2 G(t,r) -(\nu k^2-\nu/4+ikB) \bigl(-4r^{-3} \partial_r \eta + 6r^{-4}\eta\bigr), \\
         \partial_r^2 \eta(t=0) =0, \quad \partial_r^2 \eta|_{r=1}=\nu^{-1}G(t,1),\quad \partial_r^2 \eta|_{r=1+\delta}=\nu^{-1}G(t,1+\delta).
    \end{cases}
\end{equation*}
Then by virtue of  \eqref{appendix a Poincare ineq} and (\ref{appendix a time decay of eta}), we  get, by a similar derivation of (\ref{appendix a time decay of eta}), that
\begin{equation*}
\begin{split}
   &\Big\|e^{\frac{\nu t}{4C_P^2\delta^2}} \partial_t \eta\Big\|_{\LoT(L^2)}^2 +   \Big\|e^{\frac{\nu t}{4C_P^2\delta^2}}\partial_t\partial_r \eta\Big\|_{\LtT (L^2)}^2 \\
   &\qquad + \Big\|e^{\frac{\nu t}{4C_P^2\delta^2}} \partial_r^2 \eta\Big\|_{\LoT (L^2)}^2 +   \Big\|e^{\frac{\nu t}{4C_P^2\delta^2}}\partial_r^3\eta\Big\|_{\LtT( L^2)}^2\lesssim_{\nu,\delta,B,k} 1,
    \end{split}
\end{equation*}
which together with \eqref{appendix a Poincare ineq}, \eqref{appendix a time decay of eta} and the properties of $f(t,r)$ ensures that
\begin{equation*}
    \Big\|e^{\frac{\nu t}{4C_P^2\delta^2}} \partial_t\tilde{\eta}(t) \Big\|_{\LoT(L^2)} +  \Big\|e^{\frac{\nu t}{4C_P^2\delta^2}} \tilde{\eta}(t) \Big\|_{\LoT(H^2)} \lesssim_{\nu,\delta,B,k} 1.
\end{equation*}

Then by extension theorem, it's easy to extend $\tilde{\eta}$ to  a smooth function on $(t,r) \in \R_+\times [1,\oo)$ which we still  denote by $\tilde{\eta}$, and which satisfies $\textrm{supp}_r\tilde{\eta} \subset [1,2+\delta]$ and
\begin{equation}\label{estimate for tilde eta}
     \Big\|e^{\frac{\nu t}{4C_P^2\delta^2}} \partial_t \tilde{\eta}(t)\Big\|_{\LoT(L^2((1,\oo),dr))} + \Big\|e^{\frac{\nu t}{4C_P^2\delta^2}} \tilde{\eta}(t)\Big\|_{\LoT(H^2((1,\oo),dr))} <\oo.
\end{equation}

Now we are in the position to define $(w_n, f_n)$. We set $F:= \partial_t \tilde{\eta}+ \mathcal{T}\tilde{\eta} $ and define
\begin{equation*}
    w_n(t,r):=\begin{cases}
        \tilde{\eta}(t-n,r), \quad &t \geq n,\\
        0,\quad & t\leq n,
    \end{cases}
\end{equation*}
\begin{equation*}
    f_n(t,r):=\begin{cases}
        F(t-n,r), \quad & t\geq n,\\
        0,\quad & t\leq n.
   \end{cases}
\end{equation*}

It is easy to check that $(w_n,f_n)\in C^\oo_{t,x}$ solves \eqref{inhomogenous eq tc} and satisfies
$$ \textrm{supp}_r w_n \subset [1,2+\delta], \quad  \textrm{supp}_r F_n \subset [1+\delta,2+\delta].$$
Also in view of the boundedness of $a_1,a_2 $ and $\phi(r)$ on $[1,2+\delta]\subset[1,3]$ and (\ref{estimate for tilde eta}),  we can always fix some $\delta$ to be so small  that (\ref{finite energy of w_n, f_n}) holds. Since $\tilde{\eta}$ is smooth, there exists   $0<\epsilon<\delta$ so that $\| e^{t\phi(r)}\tilde{\eta}\|_{L^2((0,\oo)\times (1,1+\delta-\epsilon))} \neq 0$.  This difference in the spatial support of $w_n$ and $f_n$ leads to
\begin{equation*}
    \begin{split}
&\limsup_{n \rightarrow +\oo} \frac{\|e^{t\phi(r)}a_1(r)w_n\|_{L^2((0,\oo)\times (1,\oo))}}{\|e^{t\phi(r)} a_2(r) f_n\|_{L^2((0,\oo)\times (1,\oo))}} = \limsup_{n \rightarrow +\oo} \frac{\|e^{n\phi(r)}e^{t\phi(r)}a_1(r)\tilde{\eta}\|_{L^2((0,\oo)\times (1,2+\delta))}}{\|e^{n\phi(r)}e^{t\phi(r)} a_2(r) F\|_{L^2((0,\oo)\times (1+\delta,2+\delta))}} \\
&\qquad \geq   \frac{\min_{r\in[1,2+\delta]}a_1(r)}{\max_{r\in[1,2+\delta]}a_2(r)} \limsup_{n \rightarrow +\oo} \frac{e^{n\phi(1+\delta-\epsilon)}  \| e^{t\phi(r)}\tilde{\eta}\|_{L^2((0,\oo)\times (1,1+\delta-\epsilon))}}{e^{n\phi(1+\delta)}\|e^{t\phi(r)}  F\|_{L^2((0,\oo)\times (1+\delta,2+\delta))}} =+\oo.
    \end{split}
\end{equation*}
This completes the proof of Theorem \ref{main them-2}.
\end{proof}

\section{Proof of Theorem \ref{main thm 3} and it's applications}\label{section 4}
We first present the proof of Theorem \ref{main thm 3}.
\begin{proof}[Proof of Theorem \ref{main thm 3}]
    Since $\phi(x)$ is continuous and non-constant, there exist two open subsets $V_1, V_2\subset \Omega$, so that
    \begin{equation}
     d_1  :=\sup_{x\in V_1}\phi(x) < \inf_{x\in V_2} \phi(x) =: d_2.
    \end{equation}

We  set
$$w(0,x)=0, \quad f(s,y):=\zeta(s)\xi(y) \overline{p}_\mathcal{L}(1-s,x_0,y)\in \mathcal{C}^\oo_c(\R_+\times \Omega),$$ where  $x_0 \in V_2$ will be chosen later,
$$0\leq \zeta(s)\in \mathcal{C}^\oo_c([1/8, 7/8]), \andf  \zeta(s)=1  \ \text{on} \ [1/4,3/4],$$
and
$$0\leq \xi(y)\in \mathcal{C}^\oo_c (V_1),\quad  \xi(y) = 1 \ \text{in  some open set } \ V_3 \subset V_1. $$
Then the solution of \eqref{abstracted evolution equation} can be represented as
\begin{equation}\label{4.2}
    w(t,x) =\int_0^t \int_\Omega p_\mathcal{L}(t-s,x,y) \zeta(s)\xi(y) \overline{p}_\mathcal{L}(1-s,x_0,y) dy ds.
\end{equation}
In particular, we have
\begin{align}
   & w(t,x) = 0 \quad \text{for }\ t\leq \f18, \label{4.3}\\
& w(1,x_0) = \int_0^1 \int_\Omega  \zeta(s)\xi(y) |p_\mathcal{L}(1-s,x_0,y)|^2 dy ds \geq 0.\label{4.4}
\end{align}
Since $\text{supp}_{t,x,y} p_\mathcal{L} = \R_+\times \Omega\times\Omega$,  we observe that $p_\mathcal{L}(t,x_0,y)$ is not identically zero for $(t,x_0,y) \in [\f14,\f34]\times V_2 \times V_3$. Hence we get
$w(1,x_0) >0$, by \eqref{4.4} and choosing  $x_0\in V_2$ properly. Fixing this $x_0$, we get from the continuity of $w(t,x)$, that
\begin{equation}
    w(t,x) >0 \quad \text{for some open set } \mathcal{O} \subset [\f14,\f34]\times V_2 \ \textrm{containing} \ (1,x_0) .
\end{equation}

Now we introduce
\begin{equation*}
    w_n(t,x):=\begin{cases}
      w(t-n,x), \quad &t \geq n,\\
        0,\quad & t\leq n,
    \end{cases}
\end{equation*}
\begin{equation*}
    f_n(t,x):=\begin{cases}
        f(t-n,x), \quad & t\geq n,\\
        0,\quad & t\leq n.
   \end{cases}
\end{equation*}

In view of \eqref{4.3}, it's easy to see that so defined $(w_n,f_n)$ is the solution of \eqref{abstracted evolution equation} and satisfies \eqref{1.15}. However, there holds
\begin{equation*}
    \begin{split}
&\limsup_{n \rightarrow +\oo} \frac{\|e^{t\phi(x)}a_1(x)w_n\|_{L^2((0,\oo)\times \Omega)}}{\|e^{t\phi(x)} a_2(x) f_n\|_{L^2((0,\oo)\times \Omega)}} = \limsup_{n \rightarrow +\oo} \frac{\|e^{n\phi(x)}e^{t\phi(x)}a_1(x)w\|_{L^2((0,\oo)\times \Omega)}}{\|e^{n\phi(x)}e^{t\phi(x)} a_2(x) f\|_{L^2((\f18,\f78)\times V_1)}} \\
&\qquad \geq   \frac{\min_{x\in V_2}a_1(x)}{\max_{x\in V_1}a_2(x)} \limsup_{n \rightarrow +\oo} \frac{e^{n d_2}  \|e^{t\phi(x)} w\|_{L^2(\mathcal{O})}}{e^{n d_1}\|e^{t\phi(x)}  f\|_{L^2([\f18,\f78]\times V_1)}} =+\oo.
    \end{split}
\end{equation*}
This completes the proof of Theorem \ref{main thm 3}.
\end{proof}

As applications, we give a few examples of Theorem \ref{main thm 3}.

\no {\bf Example 4.1.} Let $\Omega=\R^n$, $\mathcal{L}=-\Delta$.

It's well known that the heat kernel on $\R^n$ is $$H(t,x,y)= \frac{1}{(4\pi t)^{\frac{n}{2}}} e^{-\frac{|x-y|^2}{4t}} >0.$$ Thus,  this example satisfies the assumptions in Theorem \ref{main thm 3}.
Here we can  replace $-\Delta$ by a second-order uniformly elliptic operator in divergence form with smooth coefficients, whose fundamental solution enjoys the following two-side estimate (see \cite{A68}):
\begin{equation}\label{two side estimates}
    c_1 t^{-\frac{n}{2}} e^{-c_2\frac{|x-y|^2}{t}} \leq p_{\mathcal{L}}(t,x,y) \leq \frac{1}{c_1} t^{-\frac{n}{2}} e^{-\frac{|x-y|^2}{c_2t}},
\end{equation}
where $c_1<1 <c_2$ are two universal constants.
Also we can extend $-\Delta$ to a family of nonlocal elliptic operators whose parabolic fundamental solution vanishes nowhere, such as $(-\Delta)^{\frac{\alpha}{2}}$. For more examples, one may check \cite{CZ16,DZZ23} and references therein.

\no {\bf Example 4.2.} Let $\Omega\subset \R^n$ be any bounded domain or exterior domain with $\pa\Omega $ smooth, and let $\mathcal{L}=-\Delta$ with domain $D(\mathcal{L})= H^2(\Omega)\cap H_0^1(\Omega)$. It's well-known that, see for example \cite{ZQ03},
\begin{equation*}
    \Big(\frac{\rho(x)}{\sqrt{t}\wedge1}\wedge1 \Big) \Big(\frac{\rho(y)}{\sqrt{t}\wedge1}\wedge1 \Big) \frac{c_1 e^{-c_2\frac{|x-y|^2}{t}}}{t^\frac{n}{2}} \leq p_{-\Delta}(t,x,y) \leq  \Big(\frac{\rho(x)}{\sqrt{t}\wedge1}\wedge1 \Big) \Big(\frac{\rho(y)}{\sqrt{t}\wedge1}\wedge1 \Big) \frac{ e^{-\frac{|x-y|^2}{c_2t}}}{c_1t^\frac{n}{2}},
\end{equation*}
where $a\wedge b =\min\{a,b\}$, $\rho(x)=\text{dist}(x,\pa \Omega)$ and $c_1<1 <c_2$ are two universal constants. Hence this example also satisfies assumptions in Theorem \ref{main thm 3}.

\no {\bf Example 4.3.} Let $\Omega=\R^n$, $\mathcal{L}=-\Delta + V$, with $V$ being real. There are lots of literatures to investigate
the following problem: under what condition of $V$, will the two-side estimates \eqref{two side estimates} holds? One may see for example \cite{CW24,LS98} and references therein.

In Theorem \ref{main them-2}, $\Omega=[1,\oo)$  and $\mathcal{T}= -\nu \Delta_{k,r} + \frac{ikB}{r^2}$. Since $\frac{ikB}{r^2}$ is complex, the fundamental solution is also complex, thus we can not expect \eqref{two side estimates} to hold anymore. However since $\frac{ikB}{r^2}\in \mathcal{C}^\oo_b((1,\oo))$ is analytic, it seems possibly that $p_{\mathcal{T}}(t,r,r')$ is analytic, by estimating $\|\nabla^\alpha_{t,r,r'} p_\mathcal{T} \|_{L^2_{t,r,r'}((a,b)\times (1,\oo) \times (1,\oo) )}$ for any $0<a<b$ and $ \alpha \in \mathbb{N}^3$, and then we may deduce  Theorem \ref{main them-2} from
Theorem \ref{main thm 3}. But this involves tedious calculations, we just end up here and shall not pursue this direction here.

\appendix

\section{Basic estimates}

\begin{lemma}\label{Appendix A1-1}Let $w\in D$, there holds
\begin{align*}
&\|w/r^{1/2}\|_{\Lor}^2 \lesssim \|w/r\|_{\Ltr}\|w'\|_{\Ltr} +\|w/r\|_{\Ltr}^2.
\end{align*}
\end{lemma}
\begin{proof}
When $w|_{r=1,\infty}=0$, it can be inferred that
\begin{align*} |w/r^{1/2}|^2&=\int_1^r\partial_s(s^{-1}|w(s)|^2)ds\\
&=\int_1^r s^{-1}(w'(s)\overline{w}(s)+w(s)\overline{w}'(s))ds- \int_1^r s^{-2}|w|^2(s) ds\\
&\lesssim \|w/r\|_{\Ltr}\|w'\|_{\Ltr} +\|w/r\|_{\Ltr}^2.
\end{align*}
\end{proof}

\begin{lemma}\label{Appendix A1-2}
Let $r_0\in(1,\oo)$ and let $0<\tilde{\delta}\ll1$ be a small constant, there holds
\begin{align*}
\int_{r_0-\tilde{\delta}r_0}^{r_0+\tilde{\delta}r_0}r^{-1}dr\lesssim\tilde{\delta}.
\end{align*}
\end{lemma}
\begin{proof}
A direct computation yields
\begin{align*}
&\int_{r_0-\tilde{\delta}r_0}^{r_0+\tilde{\delta}r_0}r^{-1}dr=\ln\Big(1+\frac{2\tilde{\delta}r_0}{r_0-\tilde{\delta}r_0}\Big)
\leq  2\tilde{\delta}(1-\tilde{\delta})^{-1}\lesssim \tilde{\delta}.
\end{align*}
\end{proof}

\section{Resolvent estimates of Couette flow}
In order to establish the decay properties of  \eqref{the linearized operator for c}, we consider
\begin{align}\label{eq:R-w-navier}
&(\mathcal{C}-ik\lambda)w=- \nu(\partial_y^2-k^2)w+ik(y-\lambda)w=F,
\end{align}
where $\lambda\in\mathbb{R}$, $k\in\Z\backslash\{0\}$, $\Omega=\mathbb{R}$ or $\Omega=[0,1]$ and $w\in D(\mathcal{C})=\bigl\{w\in \LtO: w\in H^2(\Omega)\cap H^1_0(\Omega),\ yw\in \LtO\bigr\}$.
\begin{lemma}\label{prop:R-navier-v1}
Let $k\in\Z\backslash\{0\} $, $|k|\gg \nu >0$ and $\lambda\in \R$, there holds
\begin{align*}
&(\nu k^2)^{1/3}\|w\|_{\LtO} \leq C\|F\|_{\LtO}.
\end{align*}
\end{lemma}
\begin{proof}
By integration by parts, we first get
\beno
\langle F,w\rangle_\LtO=\nu\|w'\|_{\LtO}^2+\nu k^2\|w\|_{\LtO}+ik\int_\Omega(y-\lambda)|w|^2dy.
\eeno
By taking the real part, one has
\beno
\nu\|w'\|_{\LtO}^2\leq\|F\|_{\LtO}\|w\|_{\LtO}.
\eeno

We also get by integration by parts that
\begin{align*}
&\langle F,(y-\lambda)w\rangle_\LtO=-\nu\int_\Omega w''(y-\lambda)\overline{w}dy+\nu k^2\int_\Omega(y-\lambda)|w|^2dy+ik\|(y-\lambda)w\|_{\LtO}^2\\&=\nu\int_\Omega w'\overline{w}dy+\nu\int_\Omega(y-\lambda)|w'|^2dy+\nu k^2\int_\Omega(y-\lambda)|w|^2dy+ik\|(y-\lambda)w\|_{\LtO}^2.
\end{align*}
By taking the imaginary part, one has
\begin{align*}
&|k|\|(y-\lambda)w\|_{\LtO}^2\leq \|F\|_{\LtO}\|(y-\lambda)w\|_{\LtO}+\nu \|w'\|_{\LtO}\|w\|_{\LtO},
\end{align*}
which gives
\begin{align*}
&\|(y-\lambda)w\|_{\LtO}^2\leq |k|^{-2}\|F\|_{\LtO}^2+2|k|^{-1}\nu\|w'\|_{\LtO}\|w\|_{\LtO}.
\end{align*}

Let $\delta=\nu^{1/3}|k|^{-1/3}$, $E=\Omega\cap(\lambda-\delta,\lambda+\delta)$, $E^c=\Omega\setminus(\lambda-\delta,\lambda+\delta)$. Then we have
\begin{align*}
\|w\|_{\LtO}^2
=&\|w\|_{L^2(E^c)}^2+\|w\|_{L^2(E)}^2\leq \delta^{-2}\|(y-\lambda)w\|_{\LtO}^2+2\delta\|w\|_{L^\oo(\Omega)}^2\\ \leq&\delta^{-2}|k|^{-2}\|F\|_{\LtO}^2+2\delta^{-2}|k|^{-1}\nu\|w'\|_{\LtO}\|w\|_{\LtO}+4\delta\|w'\|_{\LtO}\|w\|_{\LtO}\\
=&(\nu k^2)^{-2/3}\|F\|_{\LtO}^2+6\delta\|w'\|_{\LtO}\|w\|_{\LtO}\\
\leq& (\nu k^2)^{-2/3}\|F\|_{\LtO}^2+6\delta\nu^{-1/2}\|F\|_{\LtO}^{1/2}\|w\|_{\LtO}^{3/2},
\end{align*}
which implies
\begin{align*}
&\|w\|_{\LtO}^2
\leq C\big((\nu k^2)^{-2/3}+\delta^4\nu^{-2}\big)\|F\|_{\LtO}^2\leq C(\nu k^2)^{-2/3}\|F\|_{\LtO}^2.
\end{align*}
\end{proof}

Now we prove that the pseudospectral bound $(\nu k^2)^{1/3}$ in Lemma \ref{eq:R-w-navier} is sharp for low frequency $k$.
\begin{lemma}\label{eq:R-w-navier-2}Let $|k|=1$ and $\nu <1$. There exist $\lambda_0\in \R$, $w_0\in D(\mathcal{C})$ and $C$ independent of $\nu,\lambda_0$, so that
\begin{align*}
\|(\mathcal{C}-i\lambda)w_0\|_{\LtO}\leq C\nu^{1/3} \|w_0\|_{\LtO}.
\end{align*}
\end{lemma}
\begin{proof}Let $\lambda_0=0$ and
\begin{align*}
w_0(y)=\left\{
\begin{aligned}
&y^3(\nu^{1/3}-y)^3,\quad 0\leq y\leq \nu^{1/3},\\
&0,\quad \textrm{else}.
\end{aligned}
\right.
\end{align*}
By direct calculations, one can check that $w_0\in D(\mathcal{C})$ and
\begin{align*}
    \|w_0\|_{\LtO} \approx \nu^{\frac{13}{6}},\quad \nu\| (-\partial_y^2+1)w_0\|_{\LtO} \lesssim \nu(\nu^{-\f23}+1)\nu^\frac{13}{6}, \quad \| yw_0\|_{\LtO}\lesssim \nu^{\f13} \nu^{\frac{13}{6}}.
\end{align*}
Thus we deduce that
\begin{align*}
\|(\mathcal{C}-i\lambda)w_0\|_{\LtO} \leq
 \nu\| (-\partial_y^2+1)w_0\|_{\LtO} +\| yw_0\|_{\LtO}  \leq C\nu^{1/3} \|w_0\|_{\LtO}.
\end{align*}
\end{proof}

\section*{Acknowledgement}

 T. Li is partially supported by  National Natural Science Foundation of China under Grant 12421001.  P. Zhang is partially  supported by National Key R$\&$D Program of China under grant 2021YFA1000800 and by National Natural Science Foundation of China under Grant 12421001, 12494542 and 12288201.

\section*{Declarations}

\subsection*{Conflict of interest} The authors declare that there are no conflicts of interest.

\subsection*{Data availability}
This article has no associated data.

\end{document}